\documentclass[12pt]{article}
\usepackage{algorithm}
\usepackage{algpseudocode}
\usepackage{amsfonts}
\usepackage{amsmath}
\usepackage{amssymb} 
\usepackage{amsthm} 
\usepackage[title]{appendix} 
\usepackage{autobreak}
\usepackage{bigdelim} 
\usepackage{blkarray}
\makeatletter
\def\BA@fnsymbol#1{\ensuremath{%
  \ifcase#1\or *\or \dagger\or \ddagger\or \mathsection\or \mathparagraph\or \|\or **\or \dagger\dagger\or \ddagger\ddagger \else@ctrerr\fi}}%
\makeatother
\usepackage{bm}
\usepackage{cancel} 
\usepackage{cases}
\usepackage{chngcntr}
\usepackage{colortbl} 
\usepackage{comment}
\usepackage{empheq} 
\usepackage{enumerate}
\usepackage{fancyhdr} 
\usepackage{fancyvrb} 
\usepackage{framed} 
\usepackage{graphicx}
\usepackage{hyperref} 
\hypersetup{
  colorlinks,
  linkcolor={red!50!black},
  citecolor={blue!50!black},
  urlcolor={blue!80!black}
}
\usepackage{mathrsfs} 
\usepackage{multirow} 
\usepackage[sort&compress,numbers]{natbib}
\usepackage{setspace} 
\usepackage{subfig} 
\usepackage{pdflscape} 
\usepackage{pgfplots} 
\usepackage{pxrubrica} 
\usepackage{url}
\usepackage[colorinlistoftodos,prependcaption,textsize=footnotesize]{todonotes}

\usepackage[top=35truemm,bottom=35truemm,left=30truemm,right=30truemm]{geometry} 
\makeatletter

\@addtoreset{equation}{section}
\makeatother

\algtext*{EndWhile}
\algtext*{EndIf}
\algtext*{EndFor}

\theoremstyle{plain}
\newtheorem{theorem}{Theorem}[section]
\newtheorem{proposition}[theorem]{Proposition}%
\newtheorem{corollary}[theorem]{Corollary}
\newtheorem{lemma}[theorem]{Lemma}

\theoremstyle{definition}

\theoremstyle{remark}
\newtheorem{example}[theorem]{Example}%
\DeclareMathOperator*{\rank}{rank}

\DeclareMathOperator*{\setint}{int}
\DeclareMathOperator*{\setspan}{span}

\DeclareMathOperator*{\Ext}{Ext}

\DeclareMathOperator*{\svec}{svec}

\DeclareMathOperator*{\Ker}{Ker}

\DeclareMathOperator*{\ri}{ri}

\newcommand{\relmiddle}[1]{\mathrel{}\middle#1\mathrel{}}

\newcommand{\bbE}{\mathbb{E}}
\newcommand{\bbF}{\mathbb{F}}

\newcommand{\bbK}{\mathbb{K}}
\newcommand{\bbL}{\mathbb{L}}

\newcommand{\bbR}{\mathbb{R}}
\newcommand{\bbS}{\mathbb{S}}

\newcommand{\bbU}{\mathbb{U}}
\newcommand{\bbV}{\mathbb{V}}

\newcommand{\bbX}{\mathbb{X}}
\newcommand{\bbY}{\mathbb{Y}}

\newcommand{\calA}{\mathcal{A}}
\newcommand{\calB}{\mathcal{B}}
\newcommand{\calC}{\mathcal{C}}

\newcommand{\calH}{\mathcal{H}}

\newcommand{\calP}{\mathcal{P}}

\newcommand{\calR}{\mathcal{R}}
\newcommand{\calS}{\mathcal{S}}

\newcommand{\calV}{\mathcal{V}}

\newcommand{\CP}{\mathcal{CP}}
\newcommand{\COP}{\mathcal{COP}}

\newcommand{\RNum}[1]{\uppercase\expandafter{\romannumeral #1\relax}} 
\newcommand{\Rnum}[1]{\lowercase\expandafter{\romannumeral #1\relax}} 

\title{Facial structure of copositive and completely positive cones over a second-order cone}

\makeatletter
\let\@fnsymbol\@arabic
\makeatother

\author{
\normalsize
    Mitsuhiro Nishijima\thanks{Department of Industrial and Systems Engineering, Keio University, 3-14-1 Hiyoshi, Kohoku-ku, Yokohama-shi, 2238522, Kanagawa, Japan. ({\tt nishijima@keio.jp}).}
\and
\normalsize
        Bruno F. Louren\c{c}o\thanks{Department of Fundamental Statistical Mathematics, The Institute of Statistical Mathematics, 10-3 Midori-cho, Tachikawa-shi, 1908562, Tokyo, Japan. ({\tt bruno@ism.ac.jp}).}
        }

\begin{document}
\maketitle

\begin{abstract}\noindent
We classify the faces of copositive and completely positive cones over a second-order cone and investigate their dimension and exposedness properties.
Then we compute two parameters related to chains of faces of both cones.
At the end, we discuss some possible extensions of the results with a view toward analyzing the facial structure of general copositive and completely positive cones.
\end{abstract}
\vspace{0.5cm}


\section{Introduction}
For a closed cone $\bbK$ in the Euclidean space $\bbR^n$,
\begin{equation}
\COP(\bbK) \coloneqq \{\bm{A}\in \bbS^n \mid \bm{x}^\top\bm{A}\bm{x} \ge 0 \text{ for all $\bm{x}\in \bbK$}\} \label{eq:COP}
\end{equation}
is the \emph{copositive cone over $\bbK$} in the space $\bbS^n$ of real $n\times n$ symmetric matrices and
\begin{equation}
\CP(\bbK)\coloneqq \left\{\sum_{i=1}^k \bm{a}_i\bm{a}_i^\top \relmiddle| \text{$k$ is a positive integer and } \bm{a}_i \in \bbK \text{ for all $i = 1,\dots,k$}\right\} \label{eq:CP}
\end{equation}
is the \emph{completely positive cone over $\bbK$}.
These cones have attracted interest because they can be used to obtain convex conic reformulation of certain optimization problems with nonconvex and combinatorial constraints~\cite{BDd+2000,Burer2009,Burer2012,BD2012,BMP2016}.

If the underlying cone $\bbK$ is the nonnegative orthant $\bbR_+^n$, the cone of entrywise nonnegative vectors in $\bbR^n$, then the associated copositive and completely positive cones are the cones of copositive and completely positive matrices in the usual sense, respectively; see~\cite{SB2021} for a comprehensive discussion of them.
The nonnegative orthants are examples of \emph{symmetric cones}, i.e., self-dual and homogeneous cones.
Here we recall that a cone is said to be homogeneous if its automorphism group acts transitively on the interior of the cone.
Besides the nonnegative orthants, the class of symmetric cones includes second-order cones and cones of positive semidefinite matrices.
Motivated by a variety of applications~\cite{BD2012,BMP2016,Gowda2019} of copositive and completely positive cones over general symmetric cones,
previous works have discussed approximations~\cite{NN2024_Approximation,NN2024_Generalizations}, the membership problem~\cite{Orlitzky2021}, and their facial structure~\cite{NL2025}.

Dickinson~\cite{Dickinson2011} studied the facial structure of copositive and completely positive cones over nonnegative orthants.
He investigated extreme rays and their exposedness, and maximal faces and their dimensions of the copositive and completely positive cones.
Nishijima~\cite{Nishijima2024} computed the length of a longest chain of faces and the distance to polyhedrality of the copositive and completely positive cones over nonnegative orthants.
Our previous result~\cite{NL2025} showed that faces of a symmetric cone induce faces of the corresponding copositive cone.
In particular, by exploiting this result, we provided non-exposed extreme rays of copositive cones over symmetric cones of dimension greater than or equal to $2$~\cite[Theorem~\mbox{3.3}]{NL2025}.
This result extends the corresponding result for nonnegative orthants proved in \cite[Theorem~\mbox{4.4}]{Dickinson2011}.

The aforementioned results contribute to the understanding of the facial structure of copositive and completely positive cones over symmetric cones, but they also raise some questions.
First, we would like to obtain other classes of faces of copositive cones over symmetric cones than those provided in \cite{NL2025}.
Second, we would like to extend results for nonnegative orthants, those shown in \cite{Dickinson2011,Nishijima2024} for example, to general symmetric cones.

In this paper, we study the facial structure of copositive and completely positive cones over second-order cones.
The second-order cone $\bbL^n$ of dimension $n \ge 2$ is defined as
\begin{equation}
\bbL^n \coloneqq \left\{\bm{x}\in \bbR^n \relmiddle| x_1 \ge \sqrt{x_2^2 + \cdots +  x_n^2}\right\}. \label{eq:soc}
\end{equation}
An important motivation for this study is the fact that our previous results~\cite{NL2025} imply that copositive cones over symmetric cones have faces isomorphic to $\COP(\bbL^n)$ for some $n \ge 2$, as we will discuss in Appendix~\ref{apdx:EJA}.
In addition, as shown in Proposition~\ref{prop:clconv_cone_to_CP}, analogous results are true for completely positive cones.
Therefore, by investigating the facial structure of $\COP(\bbL^n)$ and $\CP(\bbL^n)$, we can deepen our understanding of the facial structure of copositive and completely positive cones over general symmetric cones, respectively.

One useful fact is that $\COP(\bbL^n)$ and $\CP(\bbL^n)$ can be expressed using a semidefinite constraint, respectively; see \eqref{eq:COP_Ln} and \eqref{eq:CP_Ln}.
This makes the investigation of their facial structure tractable and allows us to classify the faces of the cones and study their exposedness and dimensions as will be discussed in Sections~\ref{subsec:faces_COP}, \ref{subsec:face_CP_ge}, and \ref{subsec:face_CP_eq}.
From non-exposed faces of $\COP(\bbL^n)$, we can obtain other classes of non-exposed faces of copositive cones over symmetric cones other than those provided in \cite[Theorem~\mbox{3.3}]{NL2025}.
In addition, we compute the length of a longest chain of faces and the distance to polyhedrality of $\COP(\bbL^n)$ and $\CP(\bbL^n)$ in Sections~\ref{subsec:length_COP} and \ref{subsec:length_CP}.

Furthermore, we discuss some possible extensions.
Because second-order cones are examples of symmetric cones, we can utilize the second-order cones to probe whether results that hold for nonnegative orthants may also hold for general symmetric cones.
For example, we show that an analogous result to \cite[Theorem~\mbox{4.6.\Rnum{1}}]{Dickinson2011} does not hold for second-order cones (Example~\ref{ex:not_face_cd_dc}).
Nevertheless, a natural extension of \cite[Theorem~\mbox{4.6.\Rnum{2}}]{Dickinson2011} does indeed hold, see Proposition~\ref{prop:x_neq_Kd_cup_minusKd}.
In addition, we investigate whether results provided in this paper for $\CP(\bbL^n)$ also hold for general $\CP(\bbK)$ (Proposition~\ref{prop:exposed_face_CP_K_cap_X} and Example~\ref{ex:not_face_CP_partialK_cap_X}).

The organization of this paper is as follows.
In Section~\ref{sec:preliminaries}, we introduce the notation and technical results used in this paper.
We investigate the facial structure of $\COP(\bbL^n)$ and $\CP(\bbL^n)$ in Sections~\ref{sec:COP} and~\ref{sec:CP}, respectively.
In Section~\ref{sec:implication}, we conclude this paper by discussing possible extensions of this and the existing studies~\cite{Dickinson2011,NL2025}.

\section{Preliminaries}\label{sec:preliminaries}
\subsection{Notation}
Let $V$ be a finite-dimensional real vector space with an inner product.
For a subset $S$ of $V$, we use $\setspan S$, $S^\perp$, $\setint S$, and $\partial S$ to use the smallest subspace containing $S$, the space of elements in $V$ that are orthogonal to every element in $S$, the interior of $S$, and the boundary of $S$.
For $x \in V$, we define
\begin{align*}
\bbR_+x &\coloneqq \{\alpha x\mid \alpha \ge 0\},\\
\bbR x &\coloneqq \{\alpha x\mid \alpha \in \bbR\}.
\end{align*}

All real vectors appearing in this paper are column vectors and denoted by boldface lowercase letters such as $\bm{a}$.
For $t\in \bbR$ and $\bm{v}\in \bbR^{n-1}$, we define $(t;\bm{v}) \in \bbR^n$ as the vector obtained by concatenating $t$ and $\bm{v}$ vertically.
Following this notation, we may write $\bm{x}\in \bbR^n$ as $(x_1;\bm{x}_{2:n})$.
We use $\bm{0}$ and $\bm{e}_i$ to denote the zero vector and the vector whose $i$th element is $1$ and the others are $0$.
The $n$-dimensional Euclidean space $\bbR^n$ is equipped with the standard inner product and the induced norm denoted by $\lVert\cdot \rVert$.
For $\bm{x}\in\bbR^n$, we write $\bm{x}\bm{x}^\top \in \bbS^n$ as $\bm{x}^{\otimes 2}$ for notational convenience.
Matrices are written in boldface uppercase letters such as $\bm{A}$.
For a matrix $\bm{A}$, the $(i,j)$th element of $\bm{A}$ is written as $A_{ij}$ or $A_{i,j}$.
The rank of $\bm{A}$ is denoted by $\rank(\bm{A})$.
The zero matrix is denoted by $\bm{O}$.
The space $\bbS^n$ of real $n\times n$ symmetric matrices is equipped with the inner product defined by $\langle\bm{A},\bm{B}\rangle \coloneqq \sum_{i,j=1}^n A_{ij}B_{ij}$.
We define $T_n \coloneqq \frac{n(n+1)}{2}$, which is the dimension of $\bbS^n$.

The following lemma provides a direct sum decomposition of the Euclidean space $\bbR^n$.
Its proof is elementary, so we omit it.
\begin{lemma}\label{lem:Peirce}
Let $\bm{v}\in\mathbb{R}^{n-1}$.
Then the Euclidean space $\mathbb{R}^n$ has the following orthogonal direct sum decomposition:
\begin{equation*}
\mathbb{R}^n = \mathbb{R}(1;\bm{v}) \oplus \mathbb{R}(\|\bm{v}\|^2;-\bm{v}) \oplus
\{(0;\bm{u})\mid \bm{u}^\top\bm{v} = 0\}.
\end{equation*}
\end{lemma}

\subsection{Cones and their faces}
A subset $\bbK$ of a real inner product space $\bbV$ is called a \emph{cone} if $\alpha x \in \bbK$ for all $x\in \bbK$ and $\alpha \ge  0$.
For a cone $\bbK$, we use $\bbK^*$ to denote the set of $x\in \bbV$ such that the inner product of $x$ and $y$ is nonnegative for all $y\in \bbK$.
We call $\bbK$ \emph{pointed} if $\bbK \cap (-\bbK) = \{0\}$.
We define the dimension of $\bbK$, i.e., that of $\setspan \bbK$ as $\dim \bbK$.
Two convex cones $\bbK_1$ and $\bbK_2$ are said to be \emph{linearly isomorphic} if there exists a linear isomorphism $f\colon \setspan\bbK_1 \to \setspan\bbK_2$ such that $f(\bbK_1) = \bbK_2$.

Recall that the $n$-dimensional second-order cone $\bbL^n$ is defined as \eqref{eq:soc}.
Following the notation $\bm{x} = (x_1;\bm{x}_{2:n})$, we have
\begin{equation*}
\bbL^n = \{\bm{x}\in \bbR^n \mid x_1 \ge \lVert\bm{x}_{2:n}\rVert\}.
\end{equation*}
Its boundary can be described as
\begin{equation}
\partial\bbL^n = \{\bm{x}\in \bbR^n \mid x_1 = \lVert\bm{x}_{2:n}\rVert\}. \label{eq:bdry_Ln}
\end{equation}%
We use $\bbS_+^n$ to denote the cone of positive semidefinite matrices in $\bbS^n$ and call it the positive semidefinite cone.
Furthermore, the copositive and completely positive cones defined as \eqref{eq:COP} and \eqref{eq:CP} are indeed cones, to which we will return in Section~\ref{subsec:COP_CP}.

A nonempty convex subcone $\bbF$ of a closed convex cone $\bbK$ is termed a \emph{face} of $\bbK$ if for all $a,b\in \bbK$, if $a + b\in \bbF$, then $a,b\in \bbF$ holds.
A face $\bbF \subseteq \bbK$ is said to be \emph{maximal}, if $\bbF \neq \bbK$ and for every face $\hat \bbF$ of $\bbK$ we have that the inclusion $\bbF \subsetneq \hat \bbF$ implies that $\hat \bbF =\bbK $.
That is, there are no ``intermediary'' faces between $\bbF$ and $\bbK$.
An \emph{extreme ray} of $\bbK$ is a $1$-dimensional face of $\bbK$.
We say that an element $x \in \bbK$ \emph{generates} an extreme ray if $\bbR_+ x$ is an extreme ray of $\bbK$.
We write $\Ext \bbK$ for the set of elements generating extreme rays of $\bbK$.
A face $\bbF$ of $\bbK$ is said to be \emph{exposed} if $\bbF = \{x\}^\perp \cap \bbK$ for some $x \in \bbK^*$.
An extreme ray that is exposed is called an \emph{exposed ray}.
A \emph{chain of faces} of $\bbK$ with length $l$ is a sequence $\bbF_l \subsetneq \cdots \subsetneq \bbF_1$ such that each $\bbF_i$ is a face of $\bbK$.
Note that $\dim \bbF_l < \cdots < \dim \bbF_1$ holds since the inclusions are strict, see~\cite[Corollary~\mbox{18.1.3}]{Rockafellar1970}.
The length of a longest chain of faces of $\bbK$ is denoted by $\ell_{\bbK}$.
The \emph{distance to polyhedrality} of $\bbK$, denoted by $\ell_{\rm poly}(\bbK)$, is the maximum length of a chain of faces $\bbF_l \subsetneq \cdots \subsetneq \bbF_1$ of $\bbK$ such that the face $\bbF_l$ is not polyhedral.
Since a subface of a polyhedral cone is also polyhedral~\cite[page~\mbox{172}]{Rockafellar1970}, the non-polyhedrality of the face $\bbF_l$ is equivalent to that of all the faces $\bbF_1,\dots,\bbF_l$.
Note that any face whose dimension does not exceed $2$ is polyhedral; see, for example, \cite[Exercise~\mbox{1.65}]{SB2021}.

In this paper, we make extensive use of the facial structure of $\bbL^n$ which is summarized in the following lemma.

\begin{lemma}[{\cite[Example~3.2.3]{Pataki2000}}]\label{lem:soc_faces}
The second-order cone $\bbL^n$ has the following three types of faces.
\begin{enumerate}[(i)]
\item The cone $\bbL^n$.
\item Extreme rays generated by the elements of $\partial\bbL^n \setminus \{\bm{0}\}$.
\item The singleton $\{\bm{0}\}$.
\end{enumerate}
\end{lemma}

There is a one-to-one correspondence between the faces of the positive semidefinite cone $\bbS_+^n$ and the subspaces of $\bbR^n$, which can be described in many ways,
see \cite[Theorem~4]{BC1975}, \cite[Corollary~\mbox{12.4}]{Barvinok2002}, \cite[Section~\mbox{4.2.2}]{HM2012}, \cite[Corollary~\mbox{6.1}]{Lewis1997}, \cite[Example~4]{Orlitzky2021}, and so on.
In this paper, we adopt the representation described in \cite[Section~\mbox{4.2.2}]{HM2012} and state it in the following lemma.
For a subspace $\bbX$ of $\bbR^n$, we define
\begin{equation*}
\bbS_+(\bbX)\coloneqq\\
 \left\{\sum_{i=1}^k \bm{a}_i^{\otimes 2} \relmiddle| \text{$k$ is a positive integer and } \bm{a}_i \in \bbX \text{ for all $i = 1,\dots,k$}\right\}.
\end{equation*}

\begin{lemma}\label{lem:face_PSD}
Let $\bbX$ be a subspace of $\bbR^n$.
For any $d$-dimensional subspace $\bbY$ of $\bbX$, the set $\mathbb{S}_+(\bbY)$ is a $T_d$-dimensional face of $\mathbb{S}_+(\bbX)$ and linearly isomorphic to $\bbS_+^d$.
The isomorphism is the mapping $\phi\colon \setspan\bbS_+(\bbY) \to \bbS^d$ defined as
\begin{equation}
\phi(\bm{A}) \coloneqq \bm{P}^\top\bm{A}\bm{P}, \label{eq:isom_psd}
\end{equation}
where $\bm{P} \in \bbR^{n\times d}$ is the matrix whose columns consist of an orthonormal basis for $\bbY$.
Note that its inverse is $\phi^{-1}(\bm{S}) = \bm{P}\bm{S}\bm{P}^\top$ for each $\bm{S}\in \bbS^d$.
Conversely, for any face $\bbF$ of $\mathbb{S}_+(\mathbb{X})$, there exists a subspace $\mathbb{Y}$ of $\mathbb{X}$ such that $\bbF = \mathbb{S}_+(\mathbb{Y})$.
\end{lemma}

Because a subspace $\bbX$ of $\bbR^n$ is a closed cone, the set $\bbS_+(\bbX)$ equals the completely positive cone $\CP(\bbX)$.
However, to emphasize that $\bbS_+(\bbX)$ is a face of $\bbS_+^n$, we use the notation $\bbS_+(\bbX)$ rather than $\CP(\bbX)$.

The values $\ell_{\bbK}$ and $\ell_{\rm poly}(\bbK)$ have been computed for various cones, which we discuss in the following example.

\begin{example}\label{ex:length_faces}
For positive semidefinite cones, it is known that $\ell_{\bbS_+^n} = n + 1$~\cite[Theorem~\mbox{14}]{IL2017} and $\ell_{\rm poly}(\bbS_+^n) = n-1$~\cite[Example~1]{LMT2018}.
Recently, for the copositive cone $\COP(\bbR_+^n)$, the completely positive cone $\CP(\bbR_+^n)$, and some related cones in the space $\bbS^n$, Nishijima~\cite{Nishijima2024} showed that the length of a longest chain of faces is $T_n + 1$ and the distance to polyhedrality is $T_n-1$, which attain the worst cases possible for them.
\end{example}

As described in the following lemma, the distance to polyhedrality can be bounded in terms of the length of a longest chain of faces.

\begin{lemma}[{\cite[Theorem~\mbox{11}]{LMT2018}}]\label{lem:bound_lpolyK_lK}
If $\bbK$ is a closed convex cone that is not a subspace, then we have $\ell_{\rm poly}(\bbK) \le \ell_{\bbK} -2$.
\end{lemma}

The following lemmas concern bases for subspaces of $\bbR^n$.
First, we show that if a subspace intersects with the interior of a convex cone, then we can take a basis for the subspace such that the basis is included in the interior of the cone.

\begin{lemma}\label{lem:basis_in_K}
Let $\bbK$ be a convex cone in $\bbR^n$ and $\bbX$ be a subspace of $\bbR^n$.
If $\bbX \cap \setint\bbK \neq \emptyset$, then there exist elements in $\bbX \cap \setint\bbK$ such that they compose a basis for $\bbX$.
\end{lemma}

\begin{proof}
By the assumption of $\bbX \cap \setint\bbK \neq \emptyset$, we see that the intersection of the relative interior of $\bbX$ and that of $\setint\bbK$ is nonempty.
In addition, the sets $\bbX$ and $\setint\bbK$ are convex.
Therefore, it follows from \cite[Theorem~\mbox{3.38}]{Soltan2020} that
\begin{equation*}
\setspan(\bbX \cap \setint\bbK) = \setspan\bbX \cap \setspan(\setint\bbK) = \bbX,
\end{equation*}
so we have the desired result.
\end{proof}

Next, we show that if a subspace intersects with the interior of a pointed closed convex cone, then we can take a basis for the subspace composed of elements in the boundary of the cone.

\begin{lemma}\label{lem:basis_in_partialK}
Let $\bbK$ be a pointed closed convex cone in $\bbR^n$ and $\bbX$ be a subspace of $\bbR^n$ whose dimension is greater than or equal to $2$.
If $\bbX \cap \setint\bbK \neq \emptyset$, then there exist elements in $\bbX \cap \partial\bbK$ such that they compose a basis for $\bbX$.
\end{lemma}

\begin{proof}
First, we show that $\Ext(\bbX \cap \bbK) \subseteq \bbX \cap \partial\bbK$.
We assume that there exists $\bm{x} \in \Ext(\bbX \cap \bbK)$ such that $\bm{x} \not\in \bbX \cap \partial\bbK$.
It must follow that $\bm{x} \in \setint\bbK$.
Let $\bm{d} \in \bbX$ be a nonzero vector that is orthogonal to $\bm{x}$ and satisfies $\bm{x} \pm \bm{d} \in \bbX \cap \bbK$.
(Here, we use the assumption that $\dim\bbX \ge 2$.)
Then we have $(\bm{x} + \bm{d}) + (\bm{x} - \bm{d}) = 2\bm{x} \in \bbR_+\bm{x}$.
It follows from $\bm{x} \in \Ext(\bbX \cap \bbK)$ that $\bm{x} \pm \bm{d} \in \bbR_+\bm{x}$.
Therefore, there exists $\alpha \ge 0$ such that $\bm{x} + \bm{d} = \alpha\bm{x}$.
Taking the inner product of $\bm{d}$ with this yields $\lVert\bm{d}\rVert^2 = 0$, which is a contradiction.

Next, we show that $\setspan(\bbX \cap \partial\bbK) = \setspan(\bbX \cap \bbK)$.
Let $\bm{x} \in \setspan(\bbX \cap \bbK)$.
Then there exist a positive integer $k$, $c_1,\dots,c_k \in \bbR$, and $\bm{x}_1,\dots,\bm{x}_k \in \bbX \cap \bbK$ such that $\bm{x} = \sum_{i=1}^kc_i\bm{x}_i$.
We note that $\bbX \cap \bbK$ is a pointed closed convex cone containing points other than the origin by assumption.
Therefore, it follows from \cite[Corollary~\mbox{18.5.2}]{Rockafellar1970} that $\bbX \cap \bbK$ is the conical hull of $\Ext(\bbX \cap \bbK)$.
For each $i = 1,\dots,k$, there exist a positive integer $k_i$, $c_{i1},\dots,c_{ik_i} \in \bbR_+$, and $\bm{x}_{i1},\dots,\bm{x}_{ik_i} \in \Ext(\bbX \cap \bbK)$, which are inluded in $\bbX \cap \partial\bbK$ by the above discussion, such that $\bm{x}_i = \sum_{j=1}^{k_i}c_{ij}\bm{x}_{ij}$.
Hence, we see that
\begin{equation*}
\bm{x} = \sum_{i=1}^k\sum_{j=1}^{k_i}c_ic_{ij}\bm{x}_{ij} \in \setspan(\bbX \cap \partial\bbK).
\end{equation*}

Therefore, it follows that $\setspan(\bbX \cap \partial\bbK) = \setspan(\bbX \cap \bbK) = \bbX$, where we use the assumption and Lemma~\ref{lem:basis_in_K} to derive the second equation.
\end{proof}

\subsection{Copositive and completely positive cones}\label{subsec:COP_CP}
Recall that the copositive cone $\COP(\bbK)$ and the completely positive cone $\CP(\bbK)$ over a closed cone $\bbK$ in the Euclidean space $\bbR^n$ are denoted by \eqref{eq:COP} and \eqref{eq:CP}, respectively.
The copositive cone $\COP(\bbK)$ and the completely positive cone $\CP(\bbK)$ are dual to each other~\cite[Section~2]{SZ2003}.

In what follows, we make extensive use of the following lemma.

\begin{lemma}\label{lem:HV}
For a nonempty finite subset $\calV$ of $\bbR^n$, we define $\bm{H}(\calV) \coloneqq \sum_{\bm{h}\in \calV}\bm{h}^{\otimes 2}$.
Let $\bbK$ be a closed cone in $\bbR^n$.
Then we have the following statements.
\begin{enumerate}[(i)]
\item $\bm{H}(\calV)$ is positive semidefinite.
In particular, $\bm{H}(\calV) \in \COP(\bbK)$ holds. \label{enum:MV_psd}
\item If $\calV \subseteq \bbK$, then $\bm{H}(\calV) \in \CP(\bbK)$.\label{enum:MV_cp}
\item Let $\bm{A}\in \CP(\bbK)$.
Then $\langle\bm{A},\bm{H}(\calV)\rangle = 0$ holds if and only if $\bm{A}\in \CP(\calV^\perp \cap \bbK)$.\label{enum:AinCP_perp}
\item Let $\bm{A}\in \COP(\bbK)$, $\calV \subseteq \bbK$, and $\calV \cap \setint\bbK \neq \emptyset$.
Then $\langle\bm{A},\bm{H}(\calV)\rangle = 0$ holds if and only if $\bm{A}\in \bbS_+(\calV^\perp)$.\label{enum:AinCOP_perp}
\end{enumerate}
\end{lemma}

\begin{proof}
It follows from the definition of $\bm{H}(\calV)$ that \eqref{enum:MV_psd} and \eqref{enum:MV_cp} hold.
We prove \eqref{enum:AinCP_perp}.
By the assumption of $\bm{A}\in \CP(\bbK)$, there exist $\bm{a}_1,\dots,\bm{a}_k \in \bbK$ such that $\bm{A} = \sum_{i=1}^k\bm{a}_i^{\otimes 2}$.
If $\langle\bm{A},\bm{H}(\calV)\rangle = 0$ holds, then we have $\sum_{i=1}^k\sum_{\bm{h}\in\calV}(\bm{a}_i^\top\bm{h})^2 = 0$.
This implies that $\bm{a}_i^\top\bm{h} = 0$ for all $i = 1,\dots,k$ and $\bm{h}\in \calV$.
Therefore, $\bm{a}_i \in \calV^\perp$ for each $i = 1,\dots,k$, so we obtain $\bm{A}\in \CP(\calV^\perp \cap \bbK)$.
The converse is straightforward.

Next, we prove \eqref{enum:AinCOP_perp}.
Let $\bm{A}\in \COP(\bbK)$.
We assume that $\calV \subseteq \bbK$, $\calV \cap \setint\bbK \neq \emptyset$, and $\langle\bm{A},\bm{H}(\calV)\rangle = 0$.
Then we have
\begin{equation}
\sum_{i=1}^k\sum_{\bm{h}\in\calV}\bm{h}^\top\bm{A}\bm{h} = 0. \label{eq:A_dot_MV}
\end{equation}
Since $\bm{A} \in \COP(\bbK)$ and $\calV \subseteq \bbK$, the value $\bm{h}^\top\bm{A}\bm{h}$ is nonnegative for all $\bm{h}\in \calV$.
Combining this with \eqref{eq:A_dot_MV} yields $\bm{h}^\top\bm{A}\bm{h} = 0$ for all $\bm{h}\in \calV$.
In particular, for an arbitrary element $\bm{h}_0 \in \calV \cap \setint\bbK$, we have $\bm{h}_0^\top\bm{A}\bm{h}_0 = 0$.
This implies that $\bm{A}$ is positive semidefinite.
To see this, fix $\bm{x} \in \bbR^n$ arbitrarily.
By $\bm{h}_0 \in \setint\bbK$, there exists $\delta > 0$ such that $\bm{h}_0 \pm t\bm{x} \in \bbK$ for all $t \in (0,\delta]$.
It follows from $\bm{A} \in \COP(\bbK)$ and $\bm{h}_0^\top\bm{A}\bm{h}_0 = 0$ that
\begin{equation}
0 \le (\bm{h}_0 \pm t\bm{x})^\top\bm{A}(\bm{h}_0 \pm t\bm{x}) = \pm2t\bm{h}_0^\top\bm{A}\bm{x} + t^2\bm{x}^\top\bm{A}\bm{x}. \label{eq:proof_A_psd}
\end{equation}
Dividing both sides of \eqref{eq:proof_A_psd} by $2t$ and letting $t \downarrow 0$, we have $\pm\bm{h}_0^\top\bm{A}\bm{x} \ge 0$, i.e., $\bm{h}_0^\top\bm{A}\bm{x}$ = 0.
Substituting this into \eqref{eq:proof_A_psd} and dividing it by $t^2$, we obtain $\bm{x}^\top\bm{A}\bm{x} \ge 0$.
Since $\bm{x} \in \bbR^n$ is arbitrary, we see that $\bm{A}$ is positive semidefinite.\footnote{The proof of the positive semidefiniteness of $\bm{A}$ follows \cite[Theorem~5]{Diananda1961} and \cite[Theorem~\mbox{1.1.\Rnum{6}}]{Dickinson2013}, which proved a similar result in the case where $\bbK$ is a nonnegative orthant.}
We write $\bm{A}$ as $\sum_{i=1}^k\bm{a}_i^{\otimes 2}$ for some $\bm{a}_1,\dots,\bm{a}_k \in \bbR^n$.
Then, by \eqref{eq:A_dot_MV}, we have $\sum_{i=1}^k\sum_{\bm{h}\in \calV}(\bm{a}_i^\top\bm{h})^2 = 0$.
This implies that $\bm{a}_i \in \calV^\perp$ for each $i = 1,\dots,k$, so we obtain $\bm{A}\in \bbS_+(\calV^\perp)$.
The converse is straightforward.
\end{proof}

In the following lemma, we describe the exposed rays of $\CP(\bbK)$.

\begin{lemma}\label{lem:exp_ext_closed_cone}
Let $\mathbb{K}$ be a nonzero closed cone in $\bbR^n$.
Then the matrices $\bm{a}^{\otimes 2}$ with $\bm{a}\in\mathbb{K} \setminus \{\bm{0}\}$ are exactly the generators of extreme rays of $\CP(\bbK)$.
In addition, all the extreme rays are exposed.
\end{lemma}

\begin{proof}
The extremality of the matrices $\bm{a}^{\otimes 2}$ with $\bm{a}\in\mathbb{K} \setminus \{\bm{0}\}$ is shown in \cite[Proposition~7]{GST2013}.
The exposedness of the extreme rays can be shown in a manner similar to that used in \cite[Theorem~\mbox{4.2}]{Dickinson2011}.
\end{proof}

For certain families of cones, there are relations between the faces $\bbF$ of $\bbK$ and those of $\COP(\bbF)$ and $\CP(\bbF)$.
For example, for a face $\bbF$ of a symmetric cone $\bbK$, the set
\begin{equation}
\COP(\bbF) \cap \setspan\{\bm{f}^{\otimes 2} \mid \bm{f} \in \setspan\bbF\} \label{eq:COPF_cap_span}
\end{equation}
 is a face of $\COP(\bbK)$~\cite[Corollary~\mbox{3.2}]{NL2025}.
Moreover, we can also characterize its exposedness, which is summarized in the following theorem.

\begin{theorem}[Simplified version of Corollary~\ref{cor:COP_sym_face}]\label{thm:COP_sym_face_simplified}
Let $\bbK$ be a symmetric cone in $\bbR^n$ and $\bbF$ be a face of $\bbK$.
Then the set in \eqref{eq:COPF_cap_span} is a face of $\COP(\bbK)$.
Furthermore, it is exposed if and only if the face $\bbF$ is either the symmetric cone $\bbK$ or the singleton $\{\bm{0}\}$.
\end{theorem}

See Appendix~\ref{apdx:EJA} for the detailed statement of Theorem~\ref{thm:COP_sym_face_simplified} and its proof.
Note that this statement does not necessarily hold for general $\bbK$ as shown in Example~\ref{ex:counter_ex_general_K}.
In contrast, faces of $\bbK$ always induce exposed faces of $\CP(\bbK)$ as shown next.

\begin{proposition}\label{prop:clconv_cone_to_CP}
Let $\bbK$ be a closed convex cone and $\bbF$ be a face of $\bbK$.
Then $\CP(\bbF)$ is an exposed face of $\CP(\bbK)$.
\end{proposition}

\begin{proof}
Let $\calV$ be a basis for $(\setspan\bbF)^\perp$.
Note that $\calV^\perp = \setspan\bbF$ and $\bbF = (\setspan\bbF) \cap \bbK$.
Therefore, by \eqref{enum:AinCP_perp} of Lemma~\ref{lem:HV}, we obtain the desired result.
\end{proof}

In subsequent sections, we focus on the case where the underlying cone $\bbK$ is the second-order cone $\bbL^n$.
The S-lemma~\cite[Theorem~\mbox{2.2}]{PT2007} implies that the copositive cone $\COP(\bbL^n)$ and the completely positive cone $\CP(\bbL^n)$ can be expressed using a semidefinite constraint, respectively:
\begin{align}
\COP(\bbL^n) &= \bbS_+^n + \bbR_+\bm{J}_n, \label{eq:COP_Ln}\\
\CP(\bbL^n) &= \bbS_+^n \cap \{\bm{A}\in \bbS^n \mid \langle\bm{A},\bm{J}_n\rangle \ge 0\}, \label{eq:CP_Ln}
\end{align}
where the matrix $\bm{J}_n$ is the $n\times n$ diagonal matrix with the $(1,1)$th element $1$ and the other diagonal elements $-1$.
See also Theorem~1 and Corollary~5 of \cite{SZ2003}.
Note that all the faces of $\CP(\bbL^n)$ are exposed since the faces of $\bbS_+^n$ and the polyhedral cone $\{\bm{A} \in \bbS^n \mid \langle\bm{A},\bm{J}_n\rangle \ge 0\}$ are exposed; see \cite[Example~3.1]{BW1981_regularizing} and \cite[page~459]{Tam1976}.

As we will discuss in Appendix~\ref{apdx:EJA}, every symmetric cone $\bbK$ with $\dim \bbK \geq 2$ has a face that is linearly isomorphic to some second-order cone $\bbL^m$ with $m \geq 2$.
In view of Theorem~\ref{thm:COP_sym_face_simplified} and Proposition~\ref{prop:clconv_cone_to_CP}, this creates a strong incentive for carefully examining the facial structure of $\COP(\bbL^m)$ and $\CP(\bbL^m)$ since faces isomorphic to $\COP(\bbL^m)$ and $\CP(\bbL^m)$ appear in the facial structure of $\COP(\bbK)$ and $\CP(\bbK)$, respectively.
Moreover, because non-exposed faces of $\COP(\bbL^m)$ induce non-exposed faces of $\COP(\bbK)$, we would like to investigate the exposedness of the faces of $\COP(\bbL^m)$.

\section{Facial structure of $\COP(\bbL^n)$}\label{sec:COP}
In this section, we classify the faces of the copositive cone $\COP(\bbL^n)$ and investigate their exposedness and dimensions.
We will see that $\COP(\bbL^n)$ has the following four types of faces.
\begin{enumerate}[(i)]
\item The cone $\COP(\bbL^n)$.
\item The family of $T_d$-dimensional exposed faces of the form of $\bbS_+(\bbX)$ where $\bbX$ is a $d$-dimensional subspace of $\{(1;\bm{v})\}^\perp$ and $\bm{v} \in \bbR^{n-1}$ satisfies $\lVert\bm{v}\rVert < 1$ (Theorem~\ref{thm:v<1}). \label{enum:v<1}
\item The family of $T_d$-dimensional faces of the form of $\bbS_+(\bbX)$ where $\bbX$ is a $d$-dimensional subspace of $\{(1;\bm{v})\}^\perp$ and $\bm{v} \in \bbR^{n-1}$ satisfies $\lVert\bm{v}\rVert = 1$. Its exposedness is characterized in Theorem~\ref{thm:face_intersection_zero}.\label{enum:v=1_noJ}
\item The family of $(T_d+1)$-dimensional faces of the form of $\bbS_+(\bbX) + \bbR_+\bm{J}_n$ where $\bbX$ is a $d$-dimensional subspace of $\{(1;\bm{v})\}^\perp$ and $\bm{v} \in \bbR^{n-1}$ satisfies $\lVert\bm{v}\rVert = 1$.
Its exposedness is characterized in Theorem~\ref{thm:face_intersection_nonzero}.
\end{enumerate}
Using these results, we also compute the length of a longest chain of faces and the distance to polyhedrality of $\COP(\bbL^n)$.

To classify the faces of $\COP(\bbL^n)$, we first identify the maximal faces.
Then we describe their subfaces.
The starting point of the analysis is the observation that maximal faces of a pointed closed convex cone with nonempty interior must be exposed by extreme rays of the dual cone.
Furthermore, \emph{exposed} rays of the dual cone must expose maximal faces.
These two points are summarized in Theorems~\mbox{2.17} and \mbox{2.20} of \cite{Dickinson2011} (see also \cite[Theorem~\mbox{2.6}]{YZ2018}).
As shown in Lemma~\ref{lem:exp_ext_closed_cone}, all extreme rays of $\CP(\bbL^n)$ are exposed.
Therefore, a face of $\COP(\mathbb{L}^n)$ is maximal if and only if it can be written as $\COP(\bbL^n) \cap \{\bm{a}^{\otimes 2}\}^\perp$ for some $\bm{a} \in \bbL^n \setminus \{\bm{0}\}$.
Without loss of generality, we may assume that $a_1 = 1$, so that there exists $\bm{v} \in \bbR^{n-1}$ such that $\lVert\bm{v}\rVert \le 1$ and $\bm{a} = (1;\bm{v})$.
In the following subsection, by considering the two cases  $\|\bm{v}\| < 1$ and $\|\bm{v}\| = 1$, we classify the subfaces of the maximal face
\begin{equation}
\COP(\bbL^n) \cap \{(1;\bm{v})^{\otimes 2}\}^\perp. \label{eq:maximal_face_COP}
\end{equation}

\subsection{Faces of $\COP(\bbL^n)$}\label{subsec:faces_COP}

First, we consider the case where the vector $\bm{v}\in\mathbb{R}^{n-1}$ satisfies $\|\bm{v}\| < 1$ and provide the subfaces of the maximal face in \eqref{eq:maximal_face_COP}.

\begin{theorem}\label{thm:v<1}
Suppose that $\bm{v}\in\mathbb{R}^{n-1}$ satisfies $\|\bm{v}\| < 1$.
For any subface $\bbF$ of the maximal face in \eqref{eq:maximal_face_COP}, there exists a subspace $\bbX$ of $\{(1;\bm{v})\}^\perp$ such that
\begin{equation*}
\bbF = \mathbb{S}_+(\bbX).
\end{equation*}
Conversely, for any $d$-dimensional subspace $\bbX$ of $\{(1;\bm{v})\}^\perp$, the set $\mathbb{S}_+(\bbX)$ is a $T_d$-dimensional face of $\COP(\mathbb{L}^n)$.
Furthermore, all the subfaces of the maximal face in \eqref{eq:maximal_face_COP} are exposed faces of $\COP(\bbL^n)$.
\end{theorem}

\begin{proof}
Recall that $\COP(\bbL^n)$ can be represented as in \eqref{eq:COP_Ln}.
Since $(1;\bm{v}) \in \setint\bbL^n$, \eqref{enum:AinCOP_perp} of Lemma~\ref{lem:HV} implies that the maximal face in \eqref{eq:maximal_face_COP} is $\mathbb{S}_+(\{(1;\bm{v})\}^\perp)$.
By Lemma~\ref{lem:face_PSD}, the subfaces of the maximal face are the sets $\mathbb{S}_+(\bbX)$ where $\bbX$ is a subspace of $\{(1;\bm{v})\}^\perp$, and $\dim\mathbb{S}_+(\bbX) = T_d$ if $\dim\bbX = d$.
In what follows, we show that the subface $\mathbb{S}_+(\bbX)$ is exposed for any subspace $\bbX$ of $\{(1;\bm{v})\}^\perp$.

Since $(1;\bm{v})\in \bbX^\perp \cap \setint\bbL^n$, we see that $\bbX^\perp \cap \setint\bbL^n \neq \emptyset$.
Therefore, by Lemma~\ref{lem:basis_in_K}, we can take a basis $\calV$ for $\bbX^\perp$ such that $\calV \subseteq \bbX^\perp \cap \setint\bbL^n$.
Note that the set $\calV$ satisfies $\calV \subseteq \bbL^n$, $\calV \cap \setint\bbL^n \neq \emptyset$, and $\calV^\perp = \bbX$.
Therefore, by \eqref{enum:AinCOP_perp} of Lemma~\ref{lem:HV}, we have
$\COP(\mathbb{L}^n) \cap \{\bm{H}(\calV)\}^\perp = \mathbb{S}_+(\bbX)$.
Since $\bm{H}(\calV) \in \CP(\bbL^n)$ by \eqref{enum:MV_cp} of Lemma~\ref{lem:HV}, the face $\mathbb{S}_+(\bbX)$ is exposed.
\end{proof}

Next, we consider the case where the vector $\bm{v}\in\mathbb{R}^{n-1}$ satisfies $\|\bm{v}\| = 1$ and provide the subfaces of the maximal face in \eqref{eq:maximal_face_COP}.
Since $(1;\bm{v})^\top\bm{J}_n(1;\bm{v}) = 0$, the maximal face in \eqref{eq:maximal_face_COP} is
\begin{equation}
\mathbb{S}_+(\{(1;\bm{v})\}^\perp) + \mathbb{R}_+\bm{J}_n. \label{eq:maximal_face_COP_v=1}
\end{equation}

As shown in the following lemma, any face of the Minkowski sum of two closed convex cones can be written as the Minkowski sum of faces of the respective cones.
Analogous results for compact convex sets can be found in \cite[Theorem~\mbox{\RNum{4}.1.5.b}]{Ewald1996} and \cite[Lemma~1]{RSY2018}, for example.

\begin{lemma}\label{lem:face_intersection}
Let $\bbK_1$ and $\bbK_2$ be closed convex cones, let $\bbK \coloneqq \bbK_1 + \bbK_2$, and let $\bbF$ be a face of $\bbK$.
Then the following statements hold.
\begin{enumerate}[(i)]
\item For $i = 1,2$, the set $\bbF\cap \bbK_i$ is a face of $\bbK_i$. \label{enum:face_Ki}
\item $\bbF = (\bbF\cap \bbK_1) + (\bbF\cap \bbK_2)$ holds.
In particular, by \eqref{enum:face_Ki}, the face $\bbF$ can be written as the Minkowski sum of a face of $\bbK_1$ and that of $\bbK_2$. \label{enum:face_minkowki_sum}
\end{enumerate}
\end{lemma}

\begin{proof}
To prove \eqref{enum:face_Ki}, it suffices to show that $\bbF \cap \bbK_1$ is a face of $\bbK_1$.
The set $\bbF \cap \bbK_1$ is a nonempty convex subcone of $\bbK_1$.
For $a,b\in \bbK_1$, we assume that $a+b\in \bbF\cap \bbK_1$.
It follows from $0 \in \bbK_2$ that $a = a + 0\in \bbK$ and $b = b+0 \in \bbK$.
Since $\bbF$ is a face of $\bbK$, we have $a,b\in \bbF$, so $a,b\in \bbF \cap \bbK_1$ holds.

We next prove \eqref{enum:face_minkowki_sum}.
The inclusion ``$\supseteq$'' holds since $\bbF$ is a convex cone.
To show the reverse inclusion, let $x\in \bbF$.
Since $\bbF \subseteq \bbK$, there exist $x_1 \in \bbK_1$ and $x_2\in \bbK_2$ such that $x = x_1 + x_2 \in \bbF$.
Since $0\in \bbK_1,\bbK_2$, both $x_1$ and $x_2$ belong to $\bbK$.
Because $\bbF$ is a face of $\bbK$, we see that $x_1$ and $x_2$ belong to $\bbF$.
Therefore, $x_i \in \bbF\cap \bbK_i$ for each $i=1,2$.
\end{proof}

\begin{lemma}\label{lem:face_COP_Ln}
Let $\bbF$ be a subface  of the maximal face in \eqref{eq:maximal_face_COP_v=1}.
Then the following statements hold.
\begin{enumerate}[(i)]
\item If $\bbF \cap \mathbb{R}_+\bm{J}_n = \{\bm{O}\}$, then there exists a subspace $\bbX$ of $\{(1;\bm{v})\}^\perp$ such that $\bbF = \mathbb{S}_+(\bbX)$. \label{enum:face_intersection_zero}
\item If $\bbF \cap \mathbb{R}_+\bm{J}_n \supsetneq \{\bm{O}\}$, or equivalently, $\bbF \cap \mathbb{R}_+\bm{J}_n = \mathbb{R}_+\bm{J}_n$,
then there exists a subspace $\bbX$ of $\{(1;\bm{v})\}^\perp$ such that $\bbF = \mathbb{S}_+(\bbX) + \mathbb{R}_+\bm{J}_n$. \label{enum:face_intersection_nonzero}
\end{enumerate}
\end{lemma}

\begin{proof}
By \eqref{enum:face_minkowki_sum} of Lemma~\ref{lem:face_intersection}, we have
\begin{equation*}
\bbF =  (\bbF \cap \mathbb{S}_+(\{(1;\bm{v})\}^\perp)) + (\bbF \cap \mathbb{R}_+\bm{J}_n).
\end{equation*}
It follows from \eqref{enum:face_Ki} of Lemma~\ref{lem:face_intersection} that $\bbF \cap \mathbb{S}_+(\{(1;\bm{v})\}^\perp)$ is a face of $\mathbb{S}_+(\{(1;\bm{v})\}^\perp)$.
By Lemma~\ref{lem:face_PSD}, there exists a subspace $\bbX$ of $\{(1;\bm{v})\}^\perp$ such that $\bbF \cap \mathbb{S}_+(\{(1;\bm{v})\}^\perp) = \mathbb{S}_+(\bbX)$.
Therefore, \eqref{enum:face_intersection_zero} and \eqref{enum:face_intersection_nonzero} hold from their respective assumptions.
\end{proof}

In the upcoming two subsubsections, based on Lemma~\ref{lem:face_COP_Ln}, we consider two subcases depending on how a subface $\bbF$ of \eqref{eq:maximal_face_COP_v=1} intersects with $\mathbb{R}_+\bm{J}_n$.

\subsubsection{Case 1: \texorpdfstring{$\bbF \cap \mathbb{R}_+\bm{J}_n = \{\bm{O}\}$}{TEXT}} \label{subsec:v=1_intersection_zero}
In this case, by \eqref{enum:face_intersection_zero} of Lemma~\ref{lem:face_COP_Ln}, there exists a subspace $\bbX$ of $\{(1;\bm{v})\}^\perp$ such that $\bbF = \mathbb{S}_+(\bbX)$.
In what follows, we prove the converse: for any subspace $\bbX$ of $\{(1;\bm{v})\}^\perp$, the set $\mathbb{S}_+(\bbX)$ is a subface of the maximal face in \eqref{eq:maximal_face_COP_v=1} and, in particular, a face of $\COP(\mathbb{L}^n)$.
In addition, we show that it is exposed if and only if $(1;-\bm{v})\not\in \bbX$.

We first show that $\mathbb{S}_+(\bbX)$ is an exposed face of $\COP(\bbL^n)$ if $(1;-\bm{v})\not\in \bbX$.

\begin{lemma}\label{lem:intL_cap_Wperp_nonempty}
Let $\bbX$ be a subspace of $\{(1;\bm{v})\}^\perp$.
If $(1;-\bm{v})\not\in \bbX$, then we have $\bbX^\perp \cap \setint\bbL^n \neq \emptyset$.
\end{lemma}

\begin{proof}
Let $\bbU \coloneqq \{(0;\bm{u}) \mid \bm{u}^\top\bm{v} = 0\}$.
Then it follows that
\begin{equation*}
\bbX^\perp \cap \{(1;\bm{v})\}^\perp \subseteq \{(1;\bm{v})\}^\perp =  \bbR(1;-\bm{v}) \oplus \bbU,
\end{equation*}
where the last equality comes from Lemma~\ref{lem:Peirce}.
We claim that $\bbX^\perp \cap \{(1;\bm{v})\}^\perp \not\subseteq \bbU$.
Indeed, if $\bbX^\perp \cap \{(1;\bm{v})\}^\perp \subseteq \bbU$, the space $\bbU$ decomposes into $((\bbX^\perp \cap \{(1;\bm{v})\}^\perp)^\perp \cap \bbU) \oplus (\bbX^\perp \cap \{(1;\bm{v})\}^\perp)$.
Then we have
\begin{align}
\bbR^n &= \bbR(1;-\bm{v}) \oplus \bbU \oplus \bbR(1;\bm{v}) \nonumber\\
&= \bbR(1;-\bm{v}) \oplus ((\bbX^\perp \cap \{(1;\bm{v})\}^\perp)^\perp \cap \bbU) \oplus (\bbX^\perp \cap \{(1;\bm{v})\}^\perp) \oplus \bbR(1;\bm{v}). \label{eq:Rn_decompose_Wperp}
\end{align}
In addition, it follows that
\begin{equation}
\bbR^n = \{(1;\bm{v})\}^\perp \oplus \bbR(1;\bm{v}) = \bbX \oplus (\bbX^\perp \cap \{(1;\bm{v})\}^\perp) \oplus \bbR(1;\bm{v}). \label{eq:Rn_another_decomposition}
\end{equation}
Since \eqref{eq:Rn_decompose_Wperp} and \eqref{eq:Rn_another_decomposition} are orthogonal decompositions, we have
$\bbX = \bbR(1;-\bm{v}) \oplus ((\bbX^\perp \cap  \{(1;\bm{v})\}^\perp)^\perp \cap \bbU)$.
However, this contradicts $(1;-\bm{v}) \not\in \bbX$.

Therefore, there exists $\bm{y} \in \bbX^\perp \cap \{(1;\bm{v})\}^\perp$ such that $\bm{y} = \alpha(1;-\bm{v}) + (0;\bm{u})$, where $\alpha$ and $\bm{u}$ satisfy $\alpha \neq 0$ and $\bm{u}^\top\bm{v} = 0$.
Since $\bbX^\perp \cap \{(1;\bm{v})\}^\perp$ is a subspace, without loss of generality, we may assume that $\alpha > 0$.
Let $\beta$ be such that $\beta > \frac{\lVert\bm{u}\rVert^2}{4\alpha}$ and
\begin{equation*}
\bm{x} \coloneqq \bm{y} + \beta(1;\bm{v}) = (\alpha + \beta; (-\alpha + \beta)\bm{v} + \bm{u})\in \bbX^\perp.
\end{equation*}
Then we see that
\begin{equation*}
\lVert(-\alpha + \beta)\bm{v} + \bm{u}\rVert^2 = (-\alpha + \beta)^2 + \lVert\bm{u}\rVert^2 < (\alpha + \beta)^2.
\end{equation*}
Combining this with $\alpha + \beta > 0$, we obtain $\bm{x}\in \setint\bbL^n$.
This shows that $\bbX^\perp \cap \setint\bbL^n \neq \emptyset$.
\end{proof}

\begin{proposition}\label{prop:face_intersection_zero_exposed}
Let $\bbX$ be a subspace of $\{(1;\bm{v})\}^\perp$.
If $(1;-\bm{v})\not\in \bbX$, then $\mathbb{S}_+(\bbX)$ is an exposed face of $\COP(\bbL^n)$.
\end{proposition}

\begin{proof}
By Lemma~\ref{lem:intL_cap_Wperp_nonempty} and Lemma~\ref{lem:basis_in_K}, there exists a basis $\calV$ for $\bbX^\perp$ such that $\calV \subseteq \bbX^\perp \cap \setint\bbL^n$.
Since the set $\calV$ satisfies $\calV \subseteq \bbL^n$, $\calV \cap \setint\bbL^n \neq \emptyset$, and $\calV^\perp = \bbX$, it follows from \eqref{enum:AinCOP_perp} of Lemma~\ref{lem:HV} that $\COP(\bbL^n) \cap \{\bm{H}(\calV)\}^\perp = \bbS_+(\bbX)$.
Since $\bm{H}(\calV) \in \CP(\bbL^n)$ by \eqref{enum:MV_cp} of Lemma~\ref{lem:HV}, $\mathbb{S}_+(\bbX)$ is an exposed face of $\COP(\bbL^n)$.
\end{proof}

Next we show that $\bbS_+(\bbX)$ is a non-exposed face of $\COP(\bbL^n)$ if $(1;-\bm{v})\in\bbX$.

\begin{lemma}\label{lem:not_PSD}
Let $\bm{A}\in \mathbb{S}_+(\{(1;\bm{v})\}^\perp)$ and let $t$ be a nonzero real number.
Then the matrix $\bm{P} \coloneqq \bm{A} - t\bm{J}_n$ is not positive semidefinite.
\end{lemma}

\begin{proof}
We first consider the case where $\bm{v} = \bm{e}_1$.
Then $(1;-\bm{e}_1),(0;\bm{e}_2),\dots,(0;\bm{e}_{n-1})$ forms a basis for $\{(1;\bm{e}_1)\}^\perp$.
Since $\bm{A}\in \mathbb{S}_+(\{(1;\bm{e}_1)\}^\perp)$, there exist $\bm{a}_1,\dots,\bm{a}_k\in \{(1;\bm{e}_1)\}^\perp$ such that $\bm{A} = \sum_{i=1}^k\bm{a}_i^{\otimes 2}$.
In addition, for each $i = 1,\dots,k$, there exist $c_{i,1},\dots,c_{i,n-1}\in \bbR$ such that
\begin{equation*}
\bm{a}_i = c_{i,1}(1;-\bm{e}_1) + \sum_{j=2}^{n-1}c_{i,j}(0;\bm{e}_j) = (c_{i,1},-c_{i,1},c_{i,2},\dots,c_{i,n-1})^\top.
\end{equation*}
For brevity, we define $\bm{c}_i \coloneqq (c_{i,2},\dots,c_{i,n-1})^\top \in \bbR^{n-2}$.
Then we have
\begin{equation*}
\bm{A} = \sum_{i=1}^k\bm{a}_i^{\otimes 2} = \sum_{i=1}^k\begin{pmatrix}
c_{i,1}^2 & -c_{i,1}^2 & c_{i,1}\bm{c}_i^\top\\
-c_{i,1}^2 & c_{i,1}^2 & -c_{i,1}\bm{c}_i^\top\\
c_{i,1}\bm{c}_i & -c_{i,1}\bm{c}_i & \bm{c}_i^{\otimes 2}
\end{pmatrix}.
\end{equation*}
Letting $c \coloneqq \sum_{i=1}^kc_{i,1}^2$, we see that the upper-left corner $2\times 2$ submatrix of $\bm{P}$ is
\begin{equation*}
\begin{pmatrix}
c-t & -c\\
-c &  c+t
\end{pmatrix},
\end{equation*}
whose determinant is $-t^2 < 0$.
Therefore, $\bm{P}$ is not positive semidefinite.

Next we consider the general case.
Take a symmetric orthogonal matrix $\bm{Q}$ of order $n-1$ such that $\bm{Q}\bm{e}_1 = \bm{v}$.
For example, we can utilize a Householder matrix to obtain such $\bm{Q}$~\cite[page~\mbox{88}]{HJ2013}.
Since $\bm{A}\in \mathbb{S}_+(\{(1;\bm{v})\}^\perp)$, there exist $\bm{a}_1,\dots,\bm{a}_k\in \{(1;\bm{v})\}^\perp$ such that $\bm{A} = \sum_{i=1}^k\bm{a}_i^{\otimes 2}$.
Then we have
\begin{equation}
\begin{pmatrix}
1 & \\
 & \bm{Q}
\end{pmatrix}\bm{P}\begin{pmatrix}
1 & \\
 & \bm{Q}
\end{pmatrix}^\top =  \sum_{i=1}^k\left\{
\begin{pmatrix}
1 & \\
 & \bm{Q}
\end{pmatrix}\bm{a}_i
\right\}^{\otimes 2} - t\begin{pmatrix}
1 & \\
 & \bm{Q}
\end{pmatrix}\bm{J}_n\begin{pmatrix}
1 & \\
 & \bm{Q}
\end{pmatrix}^\top. \label{eq:P_transformation}
\end{equation}
For each $i = 1,\dots,k$, it follows from $\bm{a}_i \in \{(1;\bm{v})\}^\perp$ that
\begin{equation*}
\left\{
\begin{pmatrix}
1 & \\
 & \bm{Q}
\end{pmatrix}\bm{a}_i
\right\}^\top(1;\bm{e}_1) = \bm{a}_i^\top(1;\bm{v}) = 0.
\end{equation*}
This implies that
\begin{equation*}
\begin{pmatrix}
1 & \\
 & \bm{Q}
\end{pmatrix}\bm{a}_i \in \{(1;\bm{e}_1)\}^\perp.
\end{equation*}
In addition, we have
\begin{equation*}
\begin{pmatrix}
1 & \\
 & \bm{Q}
\end{pmatrix}\bm{J}_n\begin{pmatrix}
1 & \\
 & \bm{Q}
\end{pmatrix}^\top = \bm{J}_n.
\end{equation*}
Therefore, by the result for the case where $\bm{v} = \bm{e}_1$, we see that the matrix in \eqref{eq:P_transformation} is not positive semidefinite and neither is the matrix $\bm{P}$.
\end{proof}

\begin{lemma}\label{lem:face_intersection_zero}
The maximal face in \eqref{eq:maximal_face_COP_v=1} is the minimal exposed face of $\COP(\mathbb{L}^n)$ containing the matrix $(1;-\bm{v})^{\otimes 2}$.
\end{lemma}

\begin{proof}
Let $\bbF$ be an exposed face of $\COP(\mathbb{L}^n)$ containing $(1;-\bm{v})^{\otimes 2}$.
Then there exists $\bm{H}\in\CP(\mathbb{L}^n)$ such that $\bbF = \COP(\mathbb{L}^n) \cap \{\bm{H}\}^\perp$.
We decompose $\bm{H}$ into $\sum_{i=1}^k\bm{h}_i^{\otimes 2}$ where $\bm{h}_i \in \mathbb{L}^n$ for all $i = 1,\dots,k$.
Since $(1;-\bm{v})^{\otimes 2} \in \bbF \subseteq  \{\bm{H}\}^\perp$, it follows that
\begin{equation*}
0 = \langle (1;-\bm{v})^{\otimes 2},\bm{H}\rangle = \sum_{i=1}^k\{(1;-\bm{v})^\top\bm{h}_i\}^2.
\end{equation*}
This implies that $(1;-\bm{v})^\top\bm{h}_i = 0$ for each $i = 1,\dots,k$.
By Lemma~\ref{lem:Peirce}, there exist $\alpha_i \in \bbR$ and $\bm{u}_i \in \bbR^{n-1}$ that is orthogonal to $-\bm{v}$ such that
\begin{equation*}
\bm{h}_i = \alpha_i(1;\bm{v}) + (0;\bm{u}_i) = (\alpha_i;\alpha_i\bm{v} + \bm{u}_i).
\end{equation*}
It follows from $\bm{h}_i \in \bbL^n$ that $\alpha_i \ge 0$.
In addition, $\bm{u}_i$ equals $\bm{0}$.
Otherwise, we have $\lVert\alpha_i\bm{v} + \bm{u}_i\rVert = \sqrt{\alpha_i^2 + \lVert\bm{u}_i\rVert^2} > \alpha_i$, which contradicts $\bm{h}_i \in \bbL^n$.
Therefore, we have $\bm{h}_i = \alpha_i(1;\bm{v})$.
Letting $\alpha \coloneqq \sum_{i=1}^k\alpha_i^2$, we have
\begin{equation}
\bbF = \COP(\mathbb{L}^n) \cap \{\alpha(1;\bm{v})^{\otimes 2}\}^\perp. \label{eq:exposed_face_including_1v}
\end{equation}
If $\alpha = 0$, then \eqref{eq:exposed_face_including_1v} equals $\COP(\mathbb{L}^n)$.
Otherwise, \eqref{eq:exposed_face_including_1v} equals $\COP(\mathbb{L}^n) \cap \{(1;\bm{v})^{\otimes 2}\}^\perp$, which equals \eqref{eq:maximal_face_COP_v=1}.
\end{proof}

\begin{proposition}\label{prop:face_intersection_zero_not_exposed}
Let $\bbX$ be a subspace of $\{(1;\bm{v})\}^\perp$.
If $(1;-\bm{v})\in \bbX$, then $\mathbb{S}_+(\bbX)$ is a non-exposed face of $\COP(\bbL^n)$.
\end{proposition}

\begin{proof}
We first show that $\mathbb{S}_+(\bbX)$ is a face of $\COP(\mathbb{L}^n)$.
The set $\mathbb{S}_+(\bbX)$ is a nonempty convex subcone of $\COP(\mathbb{L}^n)$.
For $\bm{A}_1,\bm{A}_2\in \COP(\bbL^n)$, we assume that
\begin{equation}
\bm{A}_1 + \bm{A}_2 \in \mathbb{S}_+(\bbX). \label{eq:A1_plus_A2_assumption}
\end{equation}
For each $i = 1,2$, there exist $\bm{P}_i\in \bbS_+^n$ and $t_i \ge 0$ such that $\bm{A}_i = \bm{P}_i + t_i\bm{J}_n$.
Then we have
\begin{equation}
\bm{P}_1 + \bm{P}_2 =  \bm{A}_1 + \bm{A}_2 - (t_1 + t_2)\bm{J}_n. \label{eq:P1_plus_P2}
\end{equation}
If $t_1 + t_2 > 0$, since $\bm{A}_1 + \bm{A}_2 \in \mathbb{S}_+(\bbX) \subseteq \mathbb{S}_+(\{(1;\bm{v})\}^\perp)$, it follows from Lemma~\ref{lem:not_PSD} that the right-hand side of \eqref{eq:P1_plus_P2} is not positive semidefinite.
However, $\bm{P}_1$ and $\bm{P}_2$ are positive semidefinite and so is the left-hand side of \eqref{eq:P1_plus_P2}, which is a contradiction.
Therefore, combining the nonnegativity of $t_1$ and $t_2$, we see that $t_1 = t_2 = 0$, so $\bm{A}_1$ and $\bm{A}_2$ are positive semidefinite.
Since $\mathbb{S}_+(\bbX)$ is a face of $\bbS_+^n$, \eqref{eq:A1_plus_A2_assumption} implies that $\bm{A}_1,\bm{A}_2\in \bbS_+(\bbX)$.
Therefore, $\bbS_+(\bbX)$ is a face of $\COP(\bbL^n)$.

Next, we have
\begin{equation*}
(1;-\bm{v})^{\otimes 2} \in \bbS_+(\bbX) \subsetneq \bbS_+(\{(1;\bm{v})\}^\perp) + \bbR_+\bm{J}_n,
\end{equation*}
where we use the assumption that $(1;-\bm{v})\in \bbX$.
By Lemma~\ref{lem:face_intersection_zero}, the face $\bbS_+(\bbX)$ is not exposed.
\end{proof}

Combining the above results, we obtain the following theorem.

\begin{theorem}\label{thm:face_intersection_zero}
Suppose that $\bm{v} \in \bbR^{n-1}$ satisfies $\lVert\bm{v}\rVert = 1$ and $\bbF$ is a subface of the maximal face in \eqref{eq:maximal_face_COP_v=1}.
If $\bbF \cap \bbR_+\bm{J}_n = \{\bm{O}\}$, then there exists a subspace $\bbX$ of  $\{(1;\bm{v})\}^\perp$ such that \[\bbF = \bbS_+(\bbX).\]
Conversely, for any $d$-dimensional subspace $\bbX$ of $\{(1;\bm{v})\}^\perp$, the set $\bbS_+(\bbX)$ is a $T_d$-dimensional face of $\COP(\bbL^n)$.
Furthermore, it is exposed if and only if $(1;-\bm{v})\not\in \bbX$.
\end{theorem}

\begin{proof}
Recalling that we are under \eqref{enum:face_intersection_zero} of Lemma~\ref{lem:face_COP_Ln}, the statements follow from Propositions~\ref{prop:face_intersection_zero_exposed} and \ref{prop:face_intersection_zero_not_exposed}.
By Lemma~\ref{lem:face_PSD}, we see that $\dim\mathbb{S}_+(\bbX) = T_d$ if $\dim\bbX = d$.
\end{proof}

\subsubsection{Case 2: \texorpdfstring{$\bbF \cap \mathbb{R}_+\bm{J}_n \supsetneq \{\bm{O}\}$}{TEXT}}
In this case, by \eqref{enum:face_intersection_nonzero} of Lemma~\ref{lem:face_COP_Ln}, there exists a subspace $\bbX$ of $\{(1;\bm{v})\}^\perp$ such that $\bbF = \mathbb{S}_+(\bbX) + \mathbb{R}_+\bm{J}_n$.
In what follows, we show results similar to those of Section~\ref{subsec:v=1_intersection_zero}:
for any subspace $\bbX$ of $\{(1;\bm{v})\}^\perp$, the set $\mathbb{S}_+(\bbX) + \mathbb{R}_+\bm{J}_n$ is a subface of the maximal face in \eqref{eq:maximal_face_COP_v=1}.
In addition, we discuss its exposedness.

\begin{proposition}\label{prop:face_intersection_nonzero_exposed}
Let $\bbX$ be a subspace of $\{(1;\bm{v})\}^\perp$.
If $(1;-\bm{v})\not\in \bbX$, then $\mathbb{S}_+(\bbX) + \mathbb{R}_+\bm{J}_n$ is an exposed face of $\COP(\mathbb{L}^n)$.
\end{proposition}

\begin{proof}
Let $l \coloneqq \dim\bbX^\perp$.
If $l \le 1$, then $\dim\bbX \ge n-1$ and we have $\bbX = \{(1;\bm{v})\}^\perp$, which contradicts $(1;-\bm{v})\not\in \bbX$.
Therefore, we see that $l \ge 2$.
By Lemmas~\ref{lem:intL_cap_Wperp_nonempty} and \ref{lem:basis_in_partialK}, there exists a basis $\calV \coloneqq \{\bm{h}_1,\dots,\bm{h}_l\}$ for $\bbX^\perp$ such that $\calV \subseteq \bbX^\perp \cap \partial\bbL^n$.
Without loss of generality, in view of the formula for $\partial\bbL^n$ provided in \eqref{eq:bdry_Ln}, for each $j = 1,\dots,l$, we may write $\bm{h}_j$ as $(1;\bm{v}_j)$ with $\lVert\bm{v}_j\rVert = 1$.
In what follows, we show that $\bbS_+(\bbX) + \bbR_+\bm{J}_n= \COP(\bbL^n) \cap \{\bm{H}(\calV)\}^\perp$.

Let $\bm{A}\in \COP(\bbL^n) \cap \{\bm{H}(\calV)\}^\perp$.
Then, since $\bm{A}\in \COP(\bbL^n)$, there exist $\bm{P}\in \bbS_+^n$ and $t\ge 0$ such that $\bm{A} = \bm{P} + t\bm{J}_n$.
Then it follows from $\bm{A}\in \{\bm{H}(\calV)\}^\perp$ that
\begin{equation*}
0  = \langle\bm{A},\bm{H}(\calV)\rangle = \langle\bm{P},\bm{H}(\calV)\rangle,
\end{equation*}
where we use $\langle\bm{J}_n,(1;\bm{v}_j)^{\otimes 2}\rangle = 1-\lVert\bm{v}_j\rVert^2 = 0$ for all $j=1,\dots,l$ to derive the second equation.
By \eqref{enum:AinCP_perp} of Lemma~\ref{lem:HV} and $\calV^\perp = \bbX$, this implies that $\bm{P}\in \bbS_+(\bbX)$.
Therefore, we obtain the inclusion $\COP(\bbL^n) \cap \{\bm{H}(\calV)\}^\perp \subseteq \bbS_+(\bbX) + \bbR_+\bm{J}_n$.
The reverse inclusion also holds since $\bbS_+(\bbX) + \bbR_+\bm{J}_n \subseteq \COP(\bbL^n)$ and $\calV^\perp = \bbX$.
\end{proof}

\begin{lemma}\label{lem:v=1_proper_inclusion}
Let $\bbX$ be a proper subspace  of $\{(1;\bm{v})\}^\perp$.
Then it follows that $\bbS_+(\bbX) + \bbR_+\bm{J}_n \subsetneq \bbS_+(\{(1;\bm{v})\}^\perp) + \bbR_+\bm{J}_n$.
\end{lemma}

\begin{proof}
Since $\bbX \subsetneq \{(1;\bm{v})\}^\perp$, we have that $\bbR(1;\bm{v}) \subsetneq \bbX^\perp$.
In particular, we can decompose $\bbX^\perp $ as
$\bbR(1;\bm{v}) \oplus (\bbX^\perp \cap \{(1;\bm{v})\}^\perp)$ and
$\bbX^\perp \cap \{(1;\bm{v})\}^\perp \neq \{\bm{0}\}$.
Let $\bm{y}$ be a nonzero element in $\bbX^\perp \cap \{(1;\bm{v})\}^\perp$, so that $\bm{y}^{\otimes 2} \in \bbS_+(\{(1;\bm{v})\}^\perp)$.
It suffices to show that $\bm{y}^{\otimes 2} \not\in \bbS_+(\bbX) + \bbR_+\bm{J}_n$.

We assume that $\bm{y}^{\otimes 2} \in \bbS_+(\bbX) + \bbR_+\bm{J}_n$.
Then there exist a finite subset $\calV$ of $\bbX$ and $t \ge 0$ such that $\bm{y}^{\otimes 2} = \bm{H}(\calV) + t\bm{J}_n$.
Taking the inner product of $\bm{y}^{\otimes 2}$ with this yields $\lVert\bm{y}\rVert^4 = t\bm{y}^\top\bm{J}_n\bm{y}$.
Since $\bm{y}$ is nonzero, we have $t\bm{y}^\top\bm{J}_n\bm{y} > 0$.

On the other hand, since $\bm{y}\in \{(1;\bm{v})\}^\perp$, by Lemma~\ref{lem:Peirce}, there exist $\alpha \in \bbR$ and $\bm{u} \in \bbR^{n-1}$ such that $\bm{u}^\top\bm{v} = 0$ and $\bm{y} = \alpha(1;-\bm{v}) + (0;\bm{u}) = (\alpha;-\alpha\bm{v}+\bm{u})$.
Then it follows from $\lVert\bm{v}\rVert = 1$ and $\bm{u}^\top\bm{v} = 0$ that $\bm{y}^\top\bm{J}_n\bm{y} = \alpha^2 - \lVert-\alpha\bm{v} + \bm{u}\rVert^2 = -\lVert\bm{u}\rVert^2$.
Therefore, we get $t\bm{y}^\top\bm{J}_n\bm{y} \le 0$, which is a contradiction.
\end{proof}

\begin{proposition}\label{prop:face_intersection_nonzero_nonexposed}
Let $\bbX$ be a proper subspace of $\{(1;\bm{v})\}^\perp$.
If $(1;-\bm{v})\in \bbX$, then $\mathbb{S}_+(\bbX) + \mathbb{R}_+\bm{J}_n$ is a non-exposed face of $\COP(\bbL^n)$.
\end{proposition}

\begin{proof}
We first show that $\mathbb{S}_+(\bbX) + \mathbb{R}_+\bm{J}_n$ is a face of $\COP(\bbL^n)$.
The set $\mathbb{S}_+(\bbX) + \mathbb{R}_+\bm{J}_n$ is a nonempty convex subcone of $\COP(\bbL^n)$.
For any $\bm{A}_1,\bm{A}_2 \in \COP(\bbL^n)$, we assume that $\bm{A}_1 + \bm{A}_2 \in \bbS_+(\bbX) + \bbR_+\bm{J}_n$.
Therefore, there exist $\bm{P} \in \bbS_+(\bbX)$ and $t \ge 0$ such that $\bm{A}_1 + \bm{A}_2 = \bm{P} + t\bm{J}_n$.
For every $i=1,2$, since $\bm{A}_i \in \COP(\bbL^n)$, there exist $\bm{P}_i\in \bbS_+^n$ and $t_i \ge 0$ such that $\bm{A}_i = \bm{P}_i + t_i\bm{J}_n$.
Then we have
\begin{equation}
\bm{P}_1 + \bm{P}_2 = \bm{P} + (t-t_1-t_2)\bm{J}_n. \label{eq:P1_plus_P2_intersection_nonzero}
\end{equation}
If $t - t_1 -t_2 \neq 0$, by $\bm{P} \in \bbS_+(\bbX) \subseteq \bbS_+(\{(1;\bm{v})\}^\perp)$ and Lemma~\ref{lem:not_PSD}, we see that the left-hand side of \eqref{eq:P1_plus_P2_intersection_nonzero} is not positive semidefinite.
On the other hand, $\bm{P}_1 + \bm{P}_2 \in \bbS_+^n$ holds, which is a contradiction.
Therefore, $t - t_1 -t_2 = 0$ and then $\bm{P}_1 + \bm{P}_2 = \bm{P} \in \bbS_+(\bbX)$.
Since $\bm{P}_1,\bm{P}_2\in \bbS_+^n$ and $\bbS_+(\bbX)$ is a face of $\bbS_+^n$, we have $\bm{P}_1,\bm{P}_2\in \bbS_+(\bbX)$.
This implies that $\bm{A}_i \in \mathbb{S}_+(\bbX) + \mathbb{R}_+\bm{J}_n$ for each $i = 1,2$.
Thus, $\mathbb{S}_+(\bbX) + \mathbb{R}_+\bm{J}_n$ is a face of $\COP(\bbL^n)$.

Next, by the assumption on $\bbX$, it follows from Lemma~\ref{lem:v=1_proper_inclusion} that
\begin{equation*}
(1;-\bm{v})^{\otimes 2} \in \bbS_+(\bbX) + \bbR_+\bm{J}_n \subsetneq \bbS_+(\{(1;\bm{v})\}^\perp) + \bbR_+\bm{J}_n.
\end{equation*}
By Lemma~\ref{lem:face_intersection_zero}, the face $\bbS_+(\bbX) + \bbR_+\bm{J}_n$ is not exposed.
\end{proof}

Combining the above results, we obtain the following theorem.

\begin{theorem}\label{thm:face_intersection_nonzero}
Suppose that $\bm{v} \in \bbR^{n-1}$ satisfies $\lVert\bm{v}\rVert = 1$ and $\bbF$ is a subface of the maximal face in \eqref{eq:maximal_face_COP_v=1}.
If $\bbF \cap \bbR_+\bm{J}_n \supsetneq \{\bm{O}\}$, then there exists a subspace $\bbX \subseteq \{(1;\bm{v})\}^\perp$ such that
\begin{equation*}
\bbF = \bbS_+(\bbX) + \bbR_+\bm{J}_n.
\end{equation*}
Conversely, for any $d$-dimensional subspace $\bbX\subseteq  \{(1;\bm{v})\}^\perp$, the set $\bbS_+(\bbX) + \bbR_+\bm{J}_n$ is a $(T_d + 1)$-dimensional face of $\COP(\bbL^n)$.
Furthermore, it is exposed if and only if $\bbX = \{(1;\bm{v})\}^\perp$ or $(1;-\bm{v})\not\in \bbX$.
\end{theorem}

\begin{proof}
From \eqref{eq:maximal_face_COP_v=1} and Propositions~\ref{prop:face_intersection_nonzero_exposed} and \ref{prop:face_intersection_nonzero_nonexposed}, we see that $\bbS_+(\bbX) + \bbR_+\bm{J}_n$ is a face of $\COP(\bbL^n)$ for any subspace $\bbX$ of $\{(1;\bm{v})\}^\perp$.
If $\bbX = \{(1;\bm{v})\}^\perp$, the face $\bbS_+(\{(1;\bm{v})\}^\perp) + \bbR_+\bm{J}_n$ is the maximal face in \eqref{eq:maximal_face_COP_v=1}, which is exposed.
If $(1;-\bm{v})\not\in \bbX$, Proposition~\ref{prop:face_intersection_nonzero_exposed} implies that the face $\bbS_+(\bbX) + \bbR_+\bm{J}_n$ is exposed.
Otherwise, i.e., if $\bbX \subsetneq \{(1;\bm{v})\}^\perp$ and $(1;-\bm{v}) \in \bbX$, Proposition~\ref{prop:face_intersection_nonzero_nonexposed} implies that the face $\bbS_+(\bbX) + \bbR_+\bm{J}_n$ is not exposed.
When $\dim\bbX = d$, since $\bbS_+(\bbX) \subsetneq \bbS_+(\bbX) + \bbR_+\bm{J}_n$ and
\begin{equation*}
T_d = \dim\bbS_+(\bbX) <  \dim(\bbS_+(\bbX) + \bbR_+\bm{J}_n) \le \dim\bbS_+(\bbX) + \dim\bbR_+\bm{J}_n =  T_d + 1,
\end{equation*}
we obtain $\dim(\bbS_+(\bbX) + \bbR_+\bm{J}_n) = T_d + 1$.
\end{proof}

\subsection{Extreme rays of $\COP(\bbL^n)$}
A consequence of the discussion so far is that we are able to identify the extreme rays of  $\COP(\bbL^n)$.
However, regarding the four types of faces described previously,
it is not clear whether the same extreme ray can be realized uniquely as a face of the form \eqref{enum:v<1} or \eqref{enum:v=1_noJ}.
In this subsection, we clarify this situation and classify the extreme rays of $\COP(\bbL^n)$ into three disjoint families of rays.

\begin{lemma}\label{lem:exist_u_v}
Let $n\ge 3$ and let $\bm{w} \in \bbR^{n-1}$ be such that $\lVert\bm{w}\rVert > 1$.
Then there exist $\bm{u}, \bm{v} \in \bbR^{n-1}$ satisfying the following system of equations:
\begin{equation}
\begin{cases}
\bm{w} = -\bm{v} + \bm{u},\\
\bm{u}^\top\bm{v} = 0,\\
\lVert\bm{v}\rVert = 1.
\end{cases} \label{eq:exist_u_v}
\end{equation}
\end{lemma}

\begin{proof}
We see that
\begin{equation}
\eqref{eq:exist_u_v} \Longleftrightarrow \begin{cases}
\bm{w}^\top\bm{u} = \lVert\bm{u}\rVert^2,\\
\lVert\bm{u} - \bm{w}\rVert = 1,
\end{cases} \Longleftrightarrow \begin{cases}
\lVert\bm{u}\rVert = \sqrt{\lVert\bm{w}\rVert^2-1},\\
\lVert\bm{u} - \bm{w}\rVert = 1,
\end{cases} \label{eq:exist_u_v_equiv}
\end{equation}
where we replace $\bm{v}$ with $\bm{u}-\bm{w}$ to derive the first equivalence and use
\begin{equation*}
1 =  \lVert\bm{u} - \bm{w}\rVert^2 = \lVert\bm{w}\rVert^2 - 2\bm{w}^\top\bm{u} + \lVert\bm{u}\rVert^2 = \lVert\bm{w}\rVert^2 - \lVert\bm{u}\rVert^2
\end{equation*}
to derive the second equivalence.
In conclusion, it is enough to obtain a solution to the last system in
\eqref{eq:exist_u_v_equiv}, which we do next.
Let $\bm{x} \in \bbR^{n-1}$ be a vector such that $\lVert\bm{x}\rVert = 1$ and $\bm{x}^\top\bm{w} = 0$.
(Here, we use the assumption that $n \ge 3$.)
Then we see that
\begin{equation*}
\bm{u} = \left(1 - \frac{1}{\lVert\bm{w}\rVert^2}\right)\bm{w} + \sqrt{1 - \frac{1}{\lVert\bm{w}\rVert^2}}\bm{x}
\end{equation*}
is a solution of \eqref{eq:exist_u_v_equiv}.
\end{proof}

\begin{lemma}\label{lem:v<1_family}
The set
\begin{equation}
\bigcup_{\bm{v}\in \bbR^{n-1}:\lVert\bm{v}\rVert < 1}\{\bm{x}\in\bbR^n \mid \bm{x}\in\{(1;\bm{v})\}^\perp \setminus \{\bm{0}\}\} \label{eq:COP_extreme_vnorm_less_1}
\end{equation}
equals $\bbR^n \setminus (\bbL^n \cup (-\bbL^n))$.
\end{lemma}

\begin{proof}
Let $\bm{x} \in \bbR^n$ be an element of the set in \eqref{eq:COP_extreme_vnorm_less_1}.
Then there exists $\bm{v}\in \bbR^{n-1}$ with $\lVert\bm{v}\rVert < 1$ such that $\bm{x}\in\{(1;\bm{v})\}^\perp \setminus \{\bm{0}\}$.
From Lemma~\ref{lem:Peirce}, there exist $\alpha \in \bbR$ and $\bm{u}\in \bbR^{n-1}$ such that $\bm{u}^\top\bm{v} = 0$ and
\begin{equation*}
\bm{x} = \alpha(\lVert\bm{v}\rVert^2;-\bm{v}) + (0;\bm{u}) = (\alpha\lVert\bm{v}\rVert^2;-\alpha\bm{v}+\bm{u}).
\end{equation*}
We have
\begin{equation}
x_1^2 - \lVert\bm{x}_{2:n}\rVert^2 = \alpha^2\lVert\bm{v}\rVert^2(\lVert\bm{v}\rVert^2 - 1) - \lVert\bm{u}\rVert^2. \label{eq:w_in_SOC_check}
\end{equation}
Since $\bm{x} \neq \bm{0}$, this implies
\begin{equation*}
0 < \lVert\bm{x}\rVert^2 = \alpha^2\lVert\bm{v}\rVert^2(\lVert\bm{v}\rVert^2+1) + \lVert\bm{u}\rVert^2.
\end{equation*}
Therefore, at least one of the two conditions  below must hold:
\begin{enumerate}[(i)]
\item $\alpha \neq 0$ and $\bm{v} \neq \bm{0}$, \label{enum:alpha_v_neq_0}
\item $\bm{u} \neq \bm{0}$. \label{enum:u_neq_0}
\end{enumerate}
If \eqref{enum:alpha_v_neq_0} holds, then since $\alpha^2\lVert\bm{v}\rVert^2 > 0$, $\lVert\bm{v}\rVert^2 - 1 < 0$, and $\lVert\bm{u}\rVert^2 \ge 0$, \eqref{eq:w_in_SOC_check} is less than $0$.
If \eqref{enum:u_neq_0} holds, then since $\alpha^2\lVert\bm{v}\rVert^2(\lVert\bm{v}\rVert^2 - 1) \le 0$ and $\lVert\bm{u}\rVert^2 > 0$, \eqref{eq:w_in_SOC_check} is less than $0$.
Overall, we obtain $\bm{x} \not\in \bbL^n \cup (-\bbL^n)$.

Conversely, suppose that $\bm{x} \not\in \bbL^n \cup (-\bbL^n)$, i.e., $\lvert x_1 \rvert <  \lVert\bm{x}_{2:n}\rVert$.
Then the vector $\bm{v} \coloneqq -\frac{x_1}{\lVert\bm{x}_{2:n}\rVert^2}\bm{x}_{2:n}$ satisfies $\lVert\bm{v}\rVert < 1$ and $\bm{x}^\top(1;\bm{v}) = 0$.
Therefore, we see that $\bm{x}$ belongs to the set in \eqref{eq:COP_extreme_vnorm_less_1}.
\end{proof}

\begin{lemma}\label{lem:v=1_family}
Let $n\ge 3$.
Then the set
\begin{equation}
\bigcup_{\bm{v}\in \bbR^{n-1}:\lVert\bm{v}\rVert=1}\{\bm{x}\in\bbR^n \mid \bm{x}\in\{(1;\bm{v})\}^\perp \setminus \{\bm{0}\},\ (1;-\bm{v})\not\in \bbR\bm{x}\} \label{eq:COP_extreme_vnorm1}
\end{equation}
equals $\bbR^n \setminus (\bbL^n \cup (-\bbL^n))$.
\end{lemma}

\begin{proof}
Let $\bm{x} \in \bbR^n$ be an element of the set in \eqref{eq:COP_extreme_vnorm1}.
Then there exists $\bm{v}\in\bbR^{n-1}$ with $\lVert\bm{v}\rVert = 1$ such that $\bm{x}\in\{(1;\bm{v})\}^\perp \setminus \{\bm{0}\}$ and $(1;-\bm{v})\not\in \bbR\bm{x}$.
From Lemma~\ref{lem:Peirce} and $\lVert\bm{v}\rVert = 1$, there exist $\alpha \in \bbR$ and $\bm{u}\in \bbR^{n-1}$ such that $\bm{u}^\top\bm{v} = 0$ and
\begin{equation*}
\bm{x} = \alpha(1;-\bm{v}) + (0;\bm{u}) = (\alpha;-\alpha\bm{v}+\bm{u}).
\end{equation*}
Note that $\bm{u} \neq \bm{0}$ by the assumption that $(1;-\bm{v})\not\in \bbR\bm{x}$.
Then, since $\alpha^2 - \lVert-\alpha\bm{v}+\bm{u}\rVert^2 = -\lVert\bm{u}\rVert^2 < 0$, we have $\bm{x} \not\in \bbL^n \cup (-\bbL^n)$.

Conversely, suppose that $\bm{x}$ satisfies $\bm{x} \not\in \bbL^n \cup (-\bbL^n)$, i.e., $\lvert x_1\rvert < \lVert\bm{x}_{2:n}\rVert$.
We consider two cases.

If $x_1 = 0$, we let $\bm{v}\in \bbR^{n-1}$ be such that $\lVert\bm{v}\rVert = 1$ and $\bm{v}^\top\bm{x}_{2:n} = 0$.
(Here, we use the assumption that $n \ge 3$.)
Then we see that $\bm{x}\in\{(1;\bm{v})\}^\perp \setminus \{\bm{0}\}$ and $(1;-\bm{v})\not\in \bbR\bm{x}$, so $\bm{x}$ belongs to \eqref{eq:COP_extreme_vnorm1}.

If $x_1 \neq 0$, we have $\bm{x} = x_1(1;\frac{\bm{x}_{2:n}}{x_1})$ and $\lVert\frac{\bm{x}_{2:n}}{x_1}\rVert > 1$.
By $n\ge 3$ and Lemma~\ref{lem:exist_u_v}, there exist $\bm{u},\bm{v} \in \bbR^{n-1}$ such that $\frac{\bm{x}_{2:n}}{x_1} = -\bm{v} + \bm{u}$, $\bm{u}^\top\bm{v} = 0$, and $\lVert\bm{v}\rVert = 1$.
Note that $\bm{u} \neq \bm{0}$.
Then it follows that
\begin{equation*}
\bm{x} = x_1(1;-\bm{v}+\bm{u}) = x_1(1;-\bm{v}) + (0;x_1\bm{u}).
\end{equation*}
Since this implies that $\bm{x}^\top(1;\bm{v}) = 0$, we see that $\bm{x} \in \{(1;\bm{v})\}^\perp \setminus \{\bm{0}\}$.
In addition, since $x_1\bm{u} \neq \bm{0}$, we obtain $(1;-\bm{v}) \not\in \bbR\bm{x}$.
Therefore, $\bm{x}$ belongs to \eqref{eq:COP_extreme_vnorm1}.
\end{proof}

By using Lemmas~\ref{lem:v<1_family} and \ref{lem:v=1_family}, we have the following three disjoint families of extreme rays of $\COP(\bbL^n)$.

\begin{proposition}\label{prop:COP_ext}
Every extreme ray of $\COP(\bbL^n)$ falls into one of the following three families:
\begin{enumerate}[(i)]
\item $\bbR_+\bm{x}^{\otimes 2}$ with $\bm{x} \in \bbR^n \setminus (\bbL^n \cup (-\bbL^n))$, \label{enum:ext_notLn}
\item $\bbR_+\bm{J}_n$, \label{enum:ext_Jn}
\item $\bbR_+\bm{x}^{\otimes 2}$ with $\bm{x} \in \partial\bbL^n \setminus \{\bm{0}\}$. \label{enum:ext_partialLn}
\end{enumerate}
Conversely, all the rays written in the form of \eqref{enum:ext_notLn}, \eqref{enum:ext_Jn}, or \eqref{enum:ext_partialLn} are extreme rays of $\COP(\bbL^n)$.
Furthermore, \eqref{enum:ext_notLn} and \eqref{enum:ext_Jn} are exposed and \eqref{enum:ext_partialLn} is not.
\end{proposition}

\begin{proof}
Let $\calR$ be an extreme ray of $\COP(\bbL^n)$ and $\bbF$ be a maximal face of $\COP(\bbL^n)$ as in \eqref{eq:maximal_face_COP} containing $\calR$, so that there exists $\bm{v}\in \bbR^{n-1}$ such that $\lVert\bm{v}\rVert \le 1$ and
\begin{equation*}
\bbF = \COP(\bbL^n) \cap \{(1;\bm{v})^{\otimes 2}\}^\perp.
\end{equation*}

We first consider the case where $\lVert\bm{v}\rVert < 1$.
Then, by Theorem~\ref{thm:v<1} and $\dim\calR = 1$, there exists $\bm{x} \in \{(1;\bm{v})\}^\perp \setminus \{\bm{0}\}$ such that $\calR = \bbS_+(\bbR\bm{x}) = \bbR_+\bm{x}^{\otimes 2}$.
From Lemma~\ref{lem:v<1_family}, we see that $\bm{x} \in \bbR^n \setminus (\bbL^n \cup (-\bbL^n))$.
Therefore, $\calR$ falls into \eqref{enum:ext_notLn}.

Next, we consider the case where $\lVert\bm{v}\rVert = 1$ and $\calR \cap \bbR_+\bm{J}_n = \{\bm{O}\}$.
Then, by Theorem~\ref{thm:face_intersection_zero}, there exists $\bm{x} \in \{(1;\bm{v})\}^\perp \setminus \{\bm{0}\}$ such that $\calR = \bbS_+(\bbR\bm{x}) = \bbR_+\bm{x}^{\otimes 2}$.
If $(1;-\bm{v}) \in \bbR\bm{x}$, then we have $\calR = \bbR_+(1;-\bm{v})^{\otimes 2}$.
By the formula for $\partial\bbL^n$ shown in \eqref{eq:bdry_Ln}, we see that $\calR$ falls into \eqref{enum:ext_partialLn}.
Otherwise, we have $n\ge 3$ since $(1;-\bm{v}) \not\in \bbR\bm{x}$ is impossible if $n = 2$.
From Lemma~\ref{lem:v=1_family}, we see that $\bm{x} \in \bbR^n \setminus (\bbL^n \cup (-\bbL^n))$.
Therefore, $\calR$ falls into \eqref{enum:ext_notLn}.

Finally, we consider the case where $\lVert\bm{v}\rVert = 1$ and $\calR \cap \bbR_+\bm{J}_n \supsetneq \{\bm{O}\}$.
Then it follows from Theorem~\ref{thm:face_intersection_nonzero} that $\calR = \bbR_+\bm{J}_n$, which falls into \eqref{enum:ext_Jn}.

Conversely, let $\bm{x} \in \bbR^n \setminus (\bbL^n \cup (-\bbL^n))$.
By Lemma~\ref{lem:v<1_family}, there exists $\bm{v} \in \bbR^{n-1}$ such that $\lVert\bm{v}\rVert  < 1$ and $\bm{x} \in \{(1;\bm{v})\}^\perp \setminus \{\bm{0}\}$.
It follows from Theorem~\ref{thm:v<1} that $\bbS_+(\bbR\bm{x}) = \bbR_+\bm{x}^{\otimes 2}$ is an exposed ray of $\COP(\bbL^n)$.
Therefore, any ray as in \eqref{enum:ext_notLn} is an exposed ray of $\COP(\bbL^n)$.

Next, let $\bm{x} \in \partial\bbL^n\setminus \{\bm{0}\}$.
Setting $\bm{v} \coloneqq -\frac{\bm{x}_{2:n}}{x_1}$, we see that $\lVert\bm{v}\rVert = 1$, $\bbR\bm{x}$ is a subspace of $\{(1;\bm{v})\}^\perp$, and $(1;-\bm{v}) \in \bbR\bm{x}$.
It then follows from Theorem~\ref{thm:face_intersection_zero} that $\bbS_+(\bbR\bm{x}) = \bbR_+\bm{x}^{\otimes 2}$ is a non-exposed extreme ray of $\COP(\bbL^n)$.
Therefore, any ray with the form of \eqref{enum:ext_partialLn} is a non-exposed extreme ray of $\COP(\bbL^n)$.

Finally, the exposedness of $\bbR_+\bm{J}_n$ follows from Theorem~\ref{thm:face_intersection_nonzero} applied to $\bbX \coloneqq \{\bm{0}\}$.
\end{proof}

\subsection{The length of a longest chain of faces and the distance to polyhedrality of $\COP(\bbL^n)$}\label{subsec:length_COP}
In this subsection, we compute the length $\ell_{\COP(\bbL^n)}$ of a longest chain of faces and the distance $\ell_{\rm poly}(\COP(\bbL^n))$ to polyhedrality of $\COP(\bbL^n)$.

\begin{proposition}\label{prop:length_COPLn}
For $n\ge 2$, we have $\ell_{\COP(\mathbb{L}^n)} = n + 2$.
\end{proposition}

\begin{proof}
Let $\bm{v}_0 \in \mathbb{R}^{n-1}$ be any element satisfying $\|\bm{v}_0\|=1$.
Let
\begin{align*}
\bbX_{n-1} &\coloneqq \{(1;\bm{v}_0)\}^\perp = \mathbb{R}(1;-\bm{v}_0) \oplus
\{(0;\bm{u})\mid \bm{u}^\top\bm{v}_0 = 0\},\\
\bbX_{n-2} &\coloneqq \{(0;\bm{u}) \mid \bm{u}^\top\bm{v}_0 = 0\}
\end{align*}
and $\bbX_1,\dots,\bbX_{n-3}$ be subspaces of $\bbR^n$ such that $\{\bm{0}\} \subsetneq \bbX_1 \subsetneq \cdots \subsetneq \bbX_{n-3} \subsetneq \bbX_{n-2}$.
Note that $\dim\bbX_i = i$ for each $i = 1,\dots,n-1$.
Then we obtain the following chain of faces of $\COP(\mathbb{L}^n)$ with length $n + 2$:
\begin{equation}
\{\bm{O}\} \subsetneq \mathbb{S}_+(\bbX_1) \subsetneq \mathbb{S}_+(\bbX_2) \subsetneq \cdots \subsetneq \mathbb{S}_+(\bbX_{n-1}) \subsetneq \mathbb{S}_+(\bbX_{n-1}) + \mathbb{R}_+\bm{J}_n \subsetneq \COP(\mathbb{L}^n), \label{eq:longest_chain_COPLn}
\end{equation}
where Theorem~\ref{thm:face_intersection_zero} ensures that $\mathbb{S}_+(\bbX_1),\dots,\mathbb{S}_+(\bbX_{n-1})$ are faces of $\COP(\bbL^n)$ and Theorem~\ref{thm:face_intersection_nonzero} ensures that $\mathbb{S}_+(\bbX_{n-1}) + \mathbb{R}_+\bm{J}_n$ is a face of $\COP(\bbL^n)$.
\eqref{eq:longest_chain_COPLn} implies that $\ell_{\COP(\mathbb{L}^n)} \ge n + 2$.

To show $\ell_{\COP(\mathbb{L}^n)} \le n + 2$, let $\bbF_l \subsetneq \cdots \subsetneq \bbF_1$ be a chain of faces of $\COP(\mathbb{L}^n)$.
As we now consider the length of a longest chain of faces of $\COP(\mathbb{L}^n)$, without loss of generality, we assume that $\bbF_1 = \COP(\mathbb{L}^n)$ and $\bbF_2$ is a maximal face of $\COP(\mathbb{L}^n)$ as in \eqref{eq:maximal_face_COP}.
Thus, there exists $\bm{v} \in \mathbb{R}^{n-1}$ with $\|\bm{v}\| \le 1$ such that $\bbF_2 = \COP(\mathbb{L}^n) \cap \{(1;\bm{v})^{\otimes 2}\}^\perp$.

We first consider the case  $\|\bm{v}\| < 1$.
As mentioned in the proof of Theorem~\ref{thm:v<1}, $\bbF_2 = \mathbb{S}_+(\{(1;\bm{v})\}^\perp)$, which is linearly isomorphic to $\bbS_+^{n-1}$, and $\bbF_3,\dots,\bbF_l$ are faces of $\mathbb{S}_+(\{(1;\bm{v})\}^\perp)$.
It then follows from Example~\ref{ex:length_faces} that the chain $\bbF_l \subsetneq \cdots \subsetneq \bbF_2$ has length at most $n$.
Therefore, we have $l \le n + 1$.

Next, we consider the case $\|\bm{v}\| = 1$.
Let $i^* \in \{1,\dots,l\}$ be the maximum among the $i$ satisfying $\bm{J}_n \in \bbF_i$.
Then, by Theorems~\ref{thm:face_intersection_zero} and \ref{thm:face_intersection_nonzero}, the dimensions of $\bbF_2,\dots,\bbF_{i^*}$ are in $\{T_{n-1}+1,T_{n-2} + 1,\dots,T_0 +1\}$ and those of $\bbF_{i^*+1},\dots,\bbF_l$ are in $\{T_{n-1},T_{n-2},\dots,T_0\}$.
Therefore, we have $\dim\bbF_{i^*} \le T_{n+1-i^*} + 1$, $\dim\bbF_{i^*+ 1} \le T_{n+1-i^*}$, and $\dim\bbF_l \le T_{n-l+2}$.
Since $\dim\bbF_l \ge 0$, we see that $n-l+2 \ge 0$, or equivalently, $l \le n+2$.
\end{proof}

\begin{proposition}\label{prop:dpcopl}
The distance $\ell_{\rm poly}(\COP(\mathbb{L}^n))$ to polyhedrality of $\COP(\mathbb{L}^n)$ satisfies
\begin{equation*}
\ell_{\rm poly}(\COP(\mathbb{L}^n)) = \begin{cases}
1 & (n = 2),\\
n & (n \ge 3).
\end{cases}
\end{equation*}
\end{proposition}

\begin{proof}
We first consider the case $n = 2$.
The cone $\COP(\bbL^2)$ is not polyhedral since, as shown in Proposition~\ref{prop:COP_ext}, it has infinitely many extreme rays.
In addition, since the dimension of $\COP(\bbL^2)$ is $3$, any face strictly contained in $\COP(\bbL^2)$ has dimension less than or equal to $2$, so it is polyhedral.
Therefore, we have $\ell_{\rm poly}(\COP(\bbL^2)) = 1$.

We next consider the case of $n \ge 3$.
We again consider the chain~\eqref{eq:longest_chain_COPLn} and focus on the face $\mathbb{S}_+(\bbX_2)$.
Since $\dim\bbX_2 = 2$, the face $\mathbb{S}_+(\bbX_2)$ is linearly isomorphic to $\mathbb{S}_+^2$, which is not polyhedral.
Therefore, its subchain
\begin{equation*}
\mathbb{S}_+(\bbX_2) \subsetneq \cdots \subsetneq \mathbb{S}_+(\bbX_{n-1}) \subsetneq \mathbb{S}_+(\bbX_{n-1}) + \mathbb{R}_+\bm{J}_n \subsetneq \COP(\mathbb{L}^n)
\end{equation*}
consists of $n$ non-polyhedral faces of $\COP(\bbL^n)$, so we have $\ell_{\rm poly}(\COP(\bbL^n)) \ge n$.
On the other hand, by Lemma~\ref{lem:bound_lpolyK_lK} and Proposition~\ref{prop:length_COPLn}, we have
\begin{equation*}
\ell_{\rm poly}(\COP(\mathbb{L}^n)) \le \ell_{\COP(\mathbb{L}^n)} -2  = n.
\end{equation*}
Hence, we obtain the desired result.
\end{proof}

\section{Facial structure of $\CP(\bbL^n)$}\label{sec:CP}
As shown in \eqref{eq:CP_Ln}, $\CP(\bbL^n)$ can be represented as the intersection of the positive semidefinite cone $\bbS_+^n$ and the polyhedral cone $\{\bm{A} \in \bbS^n \mid \langle\bm{A},\bm{J}_n\rangle \ge 0\}$.
In general, it is known (see, for example, \cite[Theorem~4.3]{Dubins1962} and \cite[Proposition~\mbox{18}]{LMT2018}) that for given closed convex cones $\bbK_1$ and $\bbK_2$, the intersection of a face of $\bbK_1$ and that of $\bbK_2$ is a face of $\bbK_1 \cap \bbK_2$ and vice versa.
The faces of the cone $\bbS_+^n$ are given by Lemma~\ref{lem:face_PSD} and the faces of the polyhedral cone $\{\bm{A} \in \bbS^n \mid \langle\bm{A},\bm{J}_n\rangle \ge 0\}$ are obtained by replacing the symbol ``$\ge$'' with an element of $\{\ge,=\}$; see \cite[Proposition~\mbox{3.4}]{KV2018}, for example.
Therefore, the faces of $\CP(\bbL^n)$ are
\begin{equation}
\bbS_+(\bbX) \cap \{\bm{A} \in \bbS^n \mid \langle\bm{A},\bm{J}_n\rangle \mathbin{\diamondsuit} 0\}, \label{eq:face_CP}
\end{equation}
where $\bbX$ is a subspace of $\bbR^n$ and $\diamondsuit \in \{\ge,=\}$.

We can express the faces in \eqref{eq:face_CP} as completely positive cones over certain sets.
The first step towards this idea is the following lemma, see also \cite[Corollary~4]{SZ2003} for a related result.

\begin{lemma}\label{lem:rank-1_decomposition}
Let $\bbX$ be a subspace of $\bbR^n$, $\bm{A}$ be a rank-$r$ matrix in $\bbS_+(\bbX)$, and $\bm{Q} \in \bbS^n$.
Then there exist $\bm{a}_1,\dots,\bm{a}_r \in \bbX$ such that $\bm{A} = \sum_{i=1}^r\bm{a}_i^{\otimes 2}$ and $\bm{a}_i^\top\bm{Q}\bm{a}_i = \frac{\langle \bm{A},\bm{Q}\rangle}{r}$ for all $i = 1,\dots,r$.
\end{lemma}

We need a few auxiliary lemmas to prove Lemma~\ref{lem:rank-1_decomposition}.
Let $d \coloneqq \dim\bbX$, $\bm{P} \in \bbR^{n\times d}$ be the matrix whose columns consist of an orthonormal basis for $\bbX$, and $\phi \colon \setspan\bbS_+(\bbX) \to \bbS^d$ be the linear isomorphism defined as \eqref{eq:isom_psd}, i.e., $\phi(\bm{A}) = \bm{P}^\top\bm{A}\bm{P}$ for each $\bm{A} \in \setspan\bbS_+(\bbX)$.

\begin{lemma}\label{lem:rank_invariant_1}
Let $\bm{A}\in \bbS_+(\bbX)$ and $\bm{B} \in \bbR^{n\times m}$ be an arbitrary matrix where $m$ is a positive integer.
Then we have $\rank(\bm{P}^\top\bm{A}\bm{B}) = \rank(\bm{A}\bm{B})$.
\end{lemma}

\begin{proof}
By the rank--nullity theorem, it is sufficient to show that $\Ker(\bm{P}^\top\bm{A}\bm{B})$, the kernel space of $\bm{P}^\top\bm{A}\bm{B}$, is equal to $\Ker(\bm{A}\bm{B})$.
Since $\Ker(\bm{A}\bm{B}) \subseteq \Ker(\bm{P}^\top\bm{A}\bm{B})$ holds, we prove the converse inclusion.
We assume that $\bm{x} \in \bbR^m$ satisfies $\bm{P}^\top\bm{A}\bm{B}\bm{x} =  \bm{0}$.
Since $\bm{A} \in \bbS_+(\bbX)$ and $\bbS_+(\bbX) = \phi^{-1}(\bbS_+^d)$, there exists $\bm{S} \in \bbS_+^d$ such that $\bm{A} = \bm{P}\bm{S}\bm{P}^\top$.
Then we have
\begin{equation}
\bm{0} = \bm{P}^\top\bm{A}\bm{B}\bm{x} = \bm{P}^\top\bm{P}\bm{S}\bm{P}^\top\bm{B}\bm{x} = \bm{S}\bm{P}^\top\bm{B}\bm{x}. \label{eq:proof_ker_eq}
\end{equation}
Multiplying \eqref{eq:proof_ker_eq} from the left by $\bm{P}$, we have $\bm{A}\bm{B}\bm{x} = \bm{0}$.
Thus, we obtain the desired inclusion.
\end{proof}

\begin{lemma}\label{lem:rank_PtAP}
For $\bm{A} \in \bbS_+(\bbX)$, we have $\rank(\bm{A}) = \rank(\phi(\bm{A}))$.
\end{lemma}

\begin{proof}
It follows that
\begin{align*}
\rank(\phi(\bm{A})) &= \rank(\bm{P}^\top\bm{A}\bm{P})\\
&= \rank(\bm{A}\bm{P})\\
&= \rank((\bm{A}\bm{P})^\top)\\
&= \rank(\bm{P}^\top\bm{A})\\
&= \rank(\bm{A}),
\end{align*}
where we use Lemma~\ref{lem:rank_invariant_1} by setting $\bm{B}$ to $\bm{P}$ to derive the second equality, use the symmetry of $\bm{A}$ to derive the fourth equality, and use Lemma~\ref{lem:rank_invariant_1} again by setting $\bm{B}$ to the identity matrix to obtain the last equality.
\end{proof}

\begin{proof}[Proof of Lemma~\ref{lem:rank-1_decomposition}]
The proof of Lemma~\ref{lem:rank-1_decomposition} is similar to that of \cite[Corollary~4]{SZ2003}, but we provide the details here.
First, we show that there exist $\bm{a}_1,\dots,\bm{a}_r \in \bbX$ such that $\bm{A} = \sum_{i=1}^r\bm{a}_i^{\otimes 2}$.
Since $\phi(\bm{A}) \in \bbS_+^d$ and $\rank(\phi(\bm{A})) = r$ by Lemma~\ref{lem:rank_PtAP},
there exist $\bm{s}_1,\dots,\bm{s}_r \in \bbR^d$ such that $\phi(\bm{A}) = \sum_{i=1}^r\bm{s}_i^{\otimes 2}$.
Then we have
\begin{equation*}
\bm{A} = \phi^{-1}\left(\sum_{i=1}^r\bm{s}_i^{\otimes 2}\right) = \sum_{i=1}^r(\bm{P}\bm{s}_i)^{\otimes 2} \in \bbS_+(\bbX),
\end{equation*}
where the membership results from the definition of $\bm{P}$.

Next, suppose that there exists $j \in \{1,\dots,r\}$ such that  $\bm{a}_j^\top\bm{Q}\bm{a}_j \neq \frac{\langle \bm{A},\bm{Q}\rangle}{r}$.
Since all the $\bm{a}_j^\top\bm{Q}\bm{a}_j$ sum to $\langle \bm{A},\bm{Q}\rangle$,
there must exist another $k \in \{1,\dots,r\} \setminus \{j\}$ such that
\begin{equation*}
\left(\bm{a}_j^\top\bm{Q}\bm{a}_j - \frac{\langle \bm{A},\bm{Q}\rangle}{r}\right)\left(\bm{a}_k^\top\bm{Q}\bm{a}_k - \frac{\langle \bm{A},\bm{Q}\rangle}{r}\right) < 0,
\end{equation*}
that is, if $\bm{a}_j^\top\bm{Q}\bm{a}_j$ is less/greater than the average, there is another term that is, respectively, greater/less than the average.
Let $t\in \bbR$ be such that
\begin{equation*}
(\bm{a}_j + t\bm{a}_k)^\top\bm{Q}(\bm{a}_j + t\bm{a}_k) = \frac{\langle \bm{A},\bm{Q}\rangle}{r}(1 + t^2),
\end{equation*}
which has a solution since the discriminant of its associated quadratic polynomial is nonnegative.
Let
\begin{align*}
\bar{\bm{a}}_j &\coloneqq \frac{\bm{a}_j + t\bm{a}_k}{\sqrt{1 + t^2}} \in \bbX,\\
\bar{\bm{a}}_k &\coloneqq \frac{\bm{a}_k - t\bm{a}_j}{\sqrt{1 + t^2}} \in \bbX.
\end{align*}
Then since $\bar{\bm{a}}_j^{\otimes 2} + \bar{\bm{a}}_k^{\otimes 2} = \bm{a}_j^{\otimes 2} + \bm{a}_k^{\otimes 2}$, we have
\begin{equation*}
\bm{A} = \bar{\bm{a}}_j^{\otimes 2} + \bar{\bm{a}}_k^{\otimes 2}+\sum_{\substack{i=1\\i \not \in  \{j,k\}}}^r \bm{a}_i^{\otimes 2}.
\end{equation*}
In addition, it follows from the definitions of $t$ and $\bar{\bm{a}}_j$ that
\begin{equation*} \bar{\bm{a}}_j^\top\bm{Q}\bar{\bm{a}}_j = \frac{\langle \bm{A},\bm{Q}\rangle}{r}.
\end{equation*}
To sum up, if there exists $j \in \{1,\dots,r\}$ such that $\bm{a}_j^\top\bm{Q}\bm{a}_j \neq \frac{\langle \bm{A},\bm{Q}\rangle}{r}$, then by applying the above procedure, we can decrease the number of indices $i$ that do not satisfy the equation $\bm{a}_i^\top\bm{Q}\bm{a}_i = \frac{\langle \bm{A},\bm{Q}\rangle}{r}$.
Therefore, by carrying out the procedure at most $r$ times, we obtain the desired decomposition.
\end{proof}

\begin{proposition}
Let $\bbX$ be a subspace of $\bbR^n$.
Then we have the following equations:
\begin{enumerate}[(i)]
\item $\bbS_+(\bbX) \cap \{\bm{A} \in \bbS^n \mid \langle\bm{A},\bm{J}_n\rangle \ge 0\} = \CP(\bbX \cap \bbL^n)$, \label{enum:face_CP_ge}
\item $\bbS_+(\bbX) \cap \{\bm{A} \in \bbS^n \mid \langle\bm{A},\bm{J}_n\rangle = 0\} = \CP(\bbX \cap \partial\bbL^n)$. \label{enum:face_CP_eq}
\end{enumerate}
\end{proposition}

\begin{proof}
For $\bm{A} \in \bbS_+(\bbX)$, let $r$ be the rank of $\bm{A}$.
Then Lemma~\ref{lem:rank-1_decomposition} implies that there exist $\bm{a}^{(1)},\dots,\bm{a}^{(r)} \in \bbX$ such that $\bm{A} = \sum_{i=1}^r(\bm{a}^{(i)})^{\otimes 2}$ and $(\bm{a}^{(i)})^\top\bm{J}_n\bm{a}^{(i)} = \frac{\langle \bm{A},\bm{J}_n\rangle}{r}$ for all $i = 1,\dots,r$.
For each $i = 1,\dots,r$, by changing the sign of $a_1^{(i)}$ if necessary, we may assume that $a_1^{(i)} \ge 0$.
If $\langle\bm{A},\bm{J}_n\rangle \ge 0$, since $(a_1^{(i)})^2 - \lVert\bm{a}_{2:n}^{(i)}\rVert^2 \ge 0$, we have $\bm{a}^{(i)} \in \bbX \cap \bbL^n$.
Similarly, if $\langle\bm{A},\bm{J}_n\rangle = 0$, we have $\bm{a}^{(i)} \in \bbX \cap \partial\bbL^n$.
Therefore, we obtain the inclusions ``$\subseteq$'' in \eqref{enum:face_CP_ge} and \eqref{enum:face_CP_eq}.
The reverse inclusions are straightforward.
\end{proof}

To compute the dimensions of the faces in the subsequent subsections, we use the following lemma.

\begin{lemma}[{\cite[Lemma~\mbox{6.2}]{Dickinson2011}}]\label{lem:linearly_independent_matrices}
Assume that $\bm{x}_1,\dots,\bm{x}_m \in \bbR^n$ are linearly independent.
Then $(\bm{x}_i + \bm{x}_j)^{\otimes 2}$ $(1\le i\le j\le m)$ are $T_m$ linearly independent matrices.
\end{lemma}

\subsection{Dimensions of faces of the form $\CP(\bbX \cap \bbL^n)$}\label{subsec:face_CP_ge}
In this subsection, we compute the dimension of the face $\CP(\bbX \cap \bbL^n)$ for each subspace $\bbX$ of $\bbR^n$.
The dimension of $\CP(\bbX \cap \bbL^n)$ depends on how the subspace $\bbX$ intersects with the second-order cone $\bbL^n$.
We first consider the case where $\bbX \cap \setint\bbL^n \neq \emptyset$.

\begin{proposition}\label{prop:w_intLn_cap_W_nonempty}
Suppose that a $d$-dimensional subspace $\bbX$ of $\bbR^n$ satisfies $\bbX \cap \setint\bbL^n \neq \emptyset$.
Then we have $\dim\CP(\bbX \cap \bbL^n) = T_d$.
\end{proposition}

\begin{proof}
By assumption and Lemma~\ref{lem:basis_in_K}, there exist $\bm{x}_1,\dots,\bm{x}_d \in \bbX \cap \setint\bbL^n$ such that they compose a basis for $\bbX$.
Then it follows from Lemma~\ref{lem:linearly_independent_matrices} that $(\bm{x}_i+\bm{x}_j)^{\otimes 2}$ ($1\le i\le j\le d$) are linearly independent.
Since each $(\bm{x}_i+\bm{x}_j)^{\otimes 2}$ belongs to $\CP(\bbX \cap \bbL^n)$, it follows that $\dim\CP(\bbX \cap \bbL^n) \ge T_d$.
In addition, since $\CP(\bbX \cap \bbL^n)$ is included in $\bbS_+(\bbX)$ and $\dim\bbS_+(\bbX) = T_d$, it follows that $\dim\CP(\bbX \cap \bbL^n) \le T_d$.
Therefore, we have the desired result.
\end{proof}

If the subspace $\bbX$ intersects with the second-order cone $\bbL^n$ only at the origin, we see that $\CP(\bbX \cap \bbL^n) = \{\bm{O}\}$ and $\dim\CP(\bbX \cap \bbL^n) = 0$.
In what follows, we consider the case where the subspace $\bbX$ satisfies neither $\bbX \cap \setint\bbL^n \neq \emptyset$ nor $\bbX \cap \bbL^n = \{\bm{0}\}$.

\begin{proposition}\label{prop:Ln_cap_W_ray}
Suppose that a subspace $\bbX \subseteq \bbR^n$ satisfies $\bbX \cap \setint\bbL^n = \emptyset$ and $\{\bm{0}\} \subsetneq \bbX \cap \bbL^n$.
Then there exists $\bm{x} \in \partial\bbL^n \setminus \{\bm{0}\}$ such that $\bbX \cap \bbL^n = \bbR_+\bm{x}$.
Consequently, it follows that $\CP(\bbX \cap \bbL^n) = \bbR_+\bm{x}^{\otimes 2}$ and $\dim\CP(\bbX \cap \bbL^n) = 1$.
\end{proposition}

\begin{proof}
Let $\bbF$ be the minimal face of $\bbL^n$ containing $\bbX \cap \bbL^n$.
Then it follows from \cite[Proposition~3.2.2.\Rnum{3}]{Pataki2000} that
\begin{equation}
(\ri\bbF) \cap \ri(\bbX \cap \bbL^n) \neq \emptyset, \label{eq:minimal_face}
\end{equation}
where $\ri(\cdot)$ denotes the relative interior of an input set.
We recall the classification of the faces of $\bbL^n$ shown in Lemma~\ref{lem:soc_faces}.
If $\bbF=\bbL^n$, we see from \eqref{eq:minimal_face} that $\bbX \cap \bbL^n$ contains an element of $\setint\bbL^n$, which contradicts the assumption.
In addition, the case of $\bbF = \{\bm{0}\}$ implies that $\bbX \cap \bbL^n = \{\bm{0}\}$, which also contradicts the assumption.
Therefore, $\bbF$ is an extreme ray of $\bbL^n$, so there exists $\bm{x} \in \partial\bbL^n \setminus \{\bm{0}\}$ such that $\bbF = \bbR_+\bm{x}$.
Since $\ri\bbF = \{\alpha \bm{x} \mid \alpha > 0\}$ and $\bbX \cap \bbL^n$ is a closed cone, it follows from \eqref{eq:minimal_face} that $\bbR_+\bm{x} \subseteq \bbX \cap \bbL^n$.
Since $\bbF = \bbR_+\bm{x}$ and it includes $\bbX \cap \bbL^n$, we obtain $\bbX \cap \bbL^n = \bbR_+\bm{x}$.

Finally, $\CP(\bbX \cap \bbL^n) = \bbR_+\bm{x}^{\otimes 2}$ and $\dim\CP(\bbX \cap \bbL^n) = 1$ immediately follow from $\bbX \cap \bbL^n = \bbR_+\bm{x}$.
\end{proof}

\subsection{Dimensions of faces of the form $\CP(\bbX \cap \partial\bbL^n)$}\label{subsec:face_CP_eq}
In this subsection, we compute the dimension of the face $\CP(\bbX \cap \partial\bbL^n)$ for each subspace $\bbX$ of $\bbR^n$.
As in Section~\ref{subsec:face_CP_ge}, how the subspace $\bbX$ intersects with the second-order cone $\bbL^n$ determines the dimension of $\CP(\bbX \cap \partial\bbL^n)$.
We first consider the case where $\bbX \cap \setint\bbL^n \neq \emptyset$.

\begin{proposition}\label{prop:Jn_intLn_cap_W_nonempty}
Suppose that a $d$-dimensional subspace $\bbX$ of $\bbR^n$ satisfies $\bbX \cap \setint\bbL^n \neq \emptyset$.
Then we have $\dim\CP(\bbX \cap \partial\bbL^n) = T_d -1$.
\end{proposition}

\begin{proof}
By assumption, we can take $\bm{x} \in \bbX \cap \setint\bbL^n$.
Then the matrix $\bm{x}^{\otimes 2}$ verifies the strict inclusion
\begin{equation*}
\bbS_+(\bbX) \cap \{\bm{A} \in \bbS^n \mid \langle\bm{A},\bm{J}_n\rangle = 0\} \subsetneq \bbS_+(\bbX) \cap \{\bm{A} \in \bbS^n \mid \langle\bm{A},\bm{J}_n\rangle \ge 0\},
\end{equation*}
i.e., $\CP(\bbX \cap \partial\bbL^n) \subsetneq \CP(\bbX \cap \bbL^n)$.
These two sets are faces of $\CP(\bbL^n)$ and it follows from Proposition~\ref{prop:w_intLn_cap_W_nonempty} that $\dim\CP(\bbX \cap \bbL^n) = T_d$.
Therefore, we see that $\dim\CP(\bbX \cap \partial\bbL^n) \le T_d -1$.
In what follows, we provide $(T_d -1)$ linearly independent matrices belonging to $\CP(\bbX \cap \partial\bbL^n)$.

We note that the second-order cone $\bbL^n$ is homogeneous~\cite[Section~\mbox{\RNum{1}.2}]{FK1994}.
Since $(1;\bm{0}) \in \setint\bbL^n$, there exists an $n\times n$ invertible matrix $\bm{G}$ such that
$\bm{G}\bm{x} = (1;\bm{0}) \in (\bm{G}\bbX )\cap \setint\bbL^n$ and $\bm{G}(\setint\bbL^n) = \setint\bbL^n$.
We note $\bm{G}$ satisfies $\bm{G}(\partial\bbL^n) = \partial\bbL^n$.
Let $\bm{v}_1,\dots,\bm{v}_{d-1} \in \bbR^{n-1}$ be such that $(1;\bm{0}),(0;\bm{v}_1),\dots,(0;\bm{v}_{d-1})$ form an orthonormal basis for $\bm{G}\bbX$.
Then, since $\lVert\bm{v}_i\rVert = 1$ for each $i=1,\dots,d-1$, we see that
\begin{align*}
(1;\pm\bm{v}_i) &\in \bm{G}\bbX \cap \partial\bbL^n\ (i=1,\dots,d-1),\\
(\sqrt{2};\pm(\bm{v}_i+\bm{v}_j)) &\in \bm{G}\bbX \cap \partial\bbL^n\ (1\le i< j\le d-1).
\end{align*}
Therefore, we have
\begin{equation}
\begin{aligned}
\bm{G}^{-1}(1;\pm\bm{v}_i) &\in \bbX \cap \partial\bbL^n\ (i = 1,\dots,d-1),\\
\bm{G}^{-1}(\sqrt{2};\pm(\bm{v}_i+\bm{v}_j)) &\in \bbX \cap \partial\bbL^n\ (1 \le i< j \le d-1).
\end{aligned}\label{eq:vector_in_partialLn_cap_W}
\end{equation}
Using these vectors, we consider the following $(T_d -1)$ matrices:
\begin{equation}
\left\{\bm{G}^{-1}\begin{pmatrix}
1\\
\pm\bm{v}_i
\end{pmatrix}\right\}^{\otimes 2} = \bm{G}^{-1}\begin{pmatrix}
1 & \pm\bm{v}_i^\top\\
\pm\bm{v}_i &  \bm{v}_i^{\otimes 2}
\end{pmatrix}\bm{G}^{-\top} \label{eq:linear_independent_mat_i}
\end{equation}
for $i = 1,\dots,d-1$ and
\begin{multline}
\left\{\bm{G}^{-1}\begin{pmatrix}
\sqrt{2}\\
\bm{v}_i+\bm{v}_j
\end{pmatrix}\right\}^{\otimes 2} + \left\{\bm{G}^{-1}\begin{pmatrix}\sqrt{2}\\-(\bm{v}_i+\bm{v}_j)
\end{pmatrix}\right\}^{\otimes 2} =\\
\bm{G}^{-1}\begin{pmatrix}
4 & \\
 & 2(\bm{v}_i+\bm{v}_j)^{\otimes 2}
\end{pmatrix}\bm{G}^{-\top} \label{eq:linear_independent_mat_ij}
\end{multline}
for $1\le i< j\le d-1$.
It follows from \eqref{eq:vector_in_partialLn_cap_W} that the matrices in \eqref{eq:linear_independent_mat_i} and \eqref{eq:linear_independent_mat_ij} belong to $\CP(\bbX \cap \partial\bbL^n)$.
To see their linear independence, we assume that
\begin{multline*}
\sum_{i=1}^{d-1}c_i^+\bm{G}^{-1}\begin{pmatrix}
1 & \bm{v}_i^\top\\
\bm{v}_i &  \bm{v}_i^{\otimes 2}
\end{pmatrix}\bm{G}^{-\top} +
\sum_{i=1}^{d-1}c_i^-\bm{G}^{-1}\begin{pmatrix}
1 & -\bm{v}_i^\top\\
-\bm{v}_i &  \bm{v}_i^{\otimes 2}
\end{pmatrix}\bm{G}^{-\top} \\
+ \sum_{1\le i< j\le d-1}c_{ij}\bm{G}^{-1}\begin{pmatrix}
4 & \\
 & 2(\bm{v}_i+\bm{v}_j)^{\otimes 2}
\end{pmatrix}\bm{G}^{-\top} = \bm{O}.
\end{multline*}
From this, we have
\begin{multline}
\sum_{i=1}^{d-1}c_i^+\begin{pmatrix}
1 & \bm{v}_i^\top\\
\bm{v}_i &  \bm{v}_i^{\otimes 2}
\end{pmatrix} +
\sum_{i=1}^{d-1}c_i^-\begin{pmatrix}
1 & -\bm{v}_i^\top\\
-\bm{v}_i &  \bm{v}_i^{\otimes 2}
\end{pmatrix} \\
+ \sum_{1\le i< j\le d-1}c_{ij}\begin{pmatrix}
4 & \\
 & 2(\bm{v}_i+\bm{v}_j)^{\otimes 2}
\end{pmatrix} = \bm{O}.\label{eq:prove_linear_independence}
\end{multline}
It follows from the lower-left block of \eqref{eq:prove_linear_independence} that $\sum_{i=1}^{d-1}(c_i^+ - c_i^-)\bm{v}_i = \bm{0}$.
Since $\bm{v}_1,\dots,\bm{v}_{d-1}$ are linearly independent, we have $c_i^+ = c_i^-$ for all $i = 1,\dots,d-1$; hereafter, this common value is denoted by $c_i$.
With that, it follows from the lower-right block of \eqref{eq:prove_linear_independence} that
\begin{equation*}
\sum_{i=1}^{d-1}c_i\bm{v}_i^{\otimes 2} + \sum_{1\le i< j\le d-1}c_{ij}(\bm{v}_i+\bm{v}_j)^{\otimes 2} = \bm{O}.
\end{equation*}
By the linear independence of $\bm{v}_1,\dots,\bm{v}_{d-1}$ and Lemma~\ref{lem:linearly_independent_matrices}, $(\bm{v}_i+\bm{v}_j)^{\otimes 2}$ ($1\le i\le j\le d-1$) are also linearly independent, so we have $c_i = 0$ for all $i = 1,\dots,d-1$ and $c_{ij} = 0$ for all $1\le i< j\le d-1$.
\end{proof}

If the subspace $\bbX$ intersects with the second-order cone $\bbL^n$ only at the origin, we see that $\CP(\bbX \cap \partial\bbL^n) = \{\bm{O}\}$ and $\dim\CP(\bbX \cap \partial\bbL^n) = 0$.
In what follows, we consider the case where the subspace $\bbX$ satisfies neither $\bbX \cap \setint\bbL^n \neq \emptyset$ nor $\bbX \cap \bbL^n = \{\bm{0}\}$.
Then, in a manner similar to the proof of Proposition~\ref{prop:Ln_cap_W_ray}, we can show the following proposition.

\begin{proposition}\label{prop:pertialLn_cap_W_ray}
Suppose that a subspace $\bbX$ of $\bbR^n$ satisfies $\bbX \cap \setint\bbL^n = \emptyset$ and $\{\bm{0}\} \subsetneq \bbX \cap \bbL^n$.
Then there exists $\bm{x} \in \partial\bbL^n \setminus \{\bm{0}\}$ such that $\bbX \cap \partial\bbL^n = \bbR_+\bm{x}$.
Consequently, it follows that $\CP(\bbX \cap \partial\bbL^n) = \bbR_+\bm{x}^{\otimes 2}$ and $\dim\CP(\bbX \cap \partial\bbL^n) = 1$.
\end{proposition}

\subsection{The length of a longest chain of faces and the distance to polyhedrality of $\CP(\bbL^n)$}\label{subsec:length_CP}

In this subsection, we compute the length $\ell_{\CP(\bbL^n)}$ of a longest chain of faces and the distance $\ell_{\rm poly}(\CP(\bbL^n))$ to polyhedrality of $\CP(\bbL^n)$.

\begin{lemma}\label{lem:length_intersections_two_cones}
For closed convex cones $\bbK_1$ and $\bbK_2$, let $\bbK \coloneqq \bbK_1 \cap \bbK_2$.
Then we have $\ell_{\bbK} \le \ell_{\bbK_1} + \ell_{\bbK_2} - 1$.
\end{lemma}

\begin{proof}
Throughout this proof, for a finite set $I$, we use $\lvert I\rvert$ to denote its cardinality.
Let $\bbF^{(l)} \subsetneq \cdots \subsetneq \bbF^{(1)}$ be a chain of faces of $\bbK$.
For every $i$, since $\bbF^{(i)}$ is a face of $\bbK$, there exist a face $\bbF_1^{(i)}$ of $\bbK_1$ and a face $\bbF_2^{(i)}$ of $\bbK_2$ such that $\bbF^{(i)} = \bbF_1^{(i)} \cap \bbF_2^{(i)}$.
In particular, the strict inclusion $\bbF^{(i+1)} \subsetneq \bbF^{(i)}$ implies that $\bbF_1^{(i+1)}\times \bbF_2^{(i+1)} \subsetneq \bbF_1^{(i)}\times \bbF_2^{(i)}$, i.e.,
$\bbF_1^{(i+1)} \subsetneq \bbF_1^{(i)}$ or $\bbF_2^{(i+1)} \subsetneq \bbF_2^{(i)}$ holds.
We conclude that there exist a sequence $\bbF_1^{(l)} \subseteq \cdots \subseteq \bbF_1^{(1)}$ of faces of $\bbK_1$ and a sequence $\bbF_2^{(l)} \subseteq \cdots \subseteq \bbF_2^{(1)}$ of faces of $\bbK_2$ such that $\bbF^{(i)} = \bbF_1^{(i)} \cap \bbF_2^{(i)}$ and
\begin{equation}
\bbF_1^{(l)}\times \bbF_2^{(l)} \subsetneq \cdots \subsetneq \bbF_1^{(1)}\times \bbF_2^{(1)}. \label{eq:face_intersection_strict_inclusion}
\end{equation}

Now, we let $I_j \coloneqq \{i \in \{1,\dots,l-1\} \mid \bbF_j^{(i+1)} \subsetneq \bbF_j^{(i)}\}$ for $j = 1,2$.
By \eqref{eq:face_intersection_strict_inclusion}, every $i \in \{1,\dots,l-1\}$ belongs to $I_1$ or $I_2$ (it is possible that $i$ belongs to both $I_1$ and $I_2$), so we have
\begin{equation}
\lvert I_1\rvert + \lvert I_2\rvert \ge l-1. \label{eq:I1_plus_I2}
\end{equation}
We order the elements of $I_1$ as $i_1 < \dots < i_k$.
Then, since $\bbF_1^{(i_k+1)} \subsetneq \bbF_1^{(i_k)} \subsetneq \cdots \subsetneq \bbF_1^{(i_1)}$ is a chain of faces of $\bbK_1$ whose length is $\lvert I_1\rvert + 1$, it follows that
\begin{equation}
\lvert I_1\rvert + 1 \le \ell_{\bbK_1}. \label{eq:ell_K1_lower}
\end{equation}
Similarly, we see that
\begin{equation}
\lvert I_2\rvert + 1 \le \ell_{\bbK_2}. \label{eq:ell_K2_lower}
\end{equation}
By \eqref{eq:I1_plus_I2}, \eqref{eq:ell_K1_lower}, and \eqref{eq:ell_K2_lower}, we have
\begin{equation*}
l \le \lvert I_1\rvert +  \lvert I_2\rvert + 1 \le \ell_{\bbK_1} + \ell_{\bbK_2} - 1.
\end{equation*}
Since the chain $\bbF^{(l)} \subsetneq \cdots \subsetneq \bbF^{(1)}$ of faces of $\bbK$ is arbitrary, we obtain the desired result.
\end{proof}

\begin{proposition}\label{prop:lcpl}
For $n\ge 2$, we have $\ell_{\CP(\mathbb{L}^n)} = n + 2$.
\end{proposition}

\begin{proof}
Let $\{\bm{0}\} \subsetneq \bbX_1 \subsetneq \bbX_2 \subsetneq \cdots \subsetneq \bbX_{n-1} \subsetneq \bbR^n$ be a sequence of subspaces of $\bbR^n$ such that $\bbX_1 = \bbR\bm{x}$ for some $\bm{x}\in \partial\bbL^n \setminus \{\bm{0}\}$ and $\bbX_i \cap \setint\bbL^n \neq \emptyset$ for each $i = 2,\dots,n-1$.
Note that $\dim\bbX_i = i$ for each $i = 1,\dots,n-1$.
Then we have the following chain of faces of $\CP(\bbL^n)$ with length $n+2$:
\begin{multline}
\{\bm{O}\} \subsetneq \CP(\bbX_1 \cap \partial\bbL^n) \subsetneq \CP(\bbX_2 \cap \partial\bbL^n) \\\subsetneq \CP(\bbX_2 \cap \bbL^n) \subsetneq \cdots \subsetneq \CP(\bbX_{n-1} \cap \bbL^n) \subsetneq \CP(\bbL^n), \label{eq:cp_longest_chain}
\end{multline}
where we use Propositions~\ref{prop:Jn_intLn_cap_W_nonempty} and \ref{prop:pertialLn_cap_W_ray} to obtain the strict inclusions $\{\bm{O}\} \subsetneq \CP(\bbX_1 \cap \partial\bbL^n) \subsetneq \CP(\bbX_2 \cap \partial\bbL^n)$,
use Propositions~\ref{prop:w_intLn_cap_W_nonempty} and \ref{prop:Jn_intLn_cap_W_nonempty} to obtain the strict inclusion $\CP(\bbX_2 \cap \partial\bbL^n) \subsetneq \CP(\bbX_2 \cap \bbL^n)$,
and use Proposition~\ref{prop:w_intLn_cap_W_nonempty} to obtain the strict inclusions $\CP(\bbX_2 \cap \bbL^n) \subsetneq \cdots \subsetneq \CP(\bbX_{n-1} \cap \bbL^n) \subsetneq \CP(\bbL^n)$.
Therefore, we have $\ell_{\CP(\bbL^n)} \ge n+2$.
In addition, since the length of a longest chain of faces of $\bbS_+^n$ is $n+1$ (Example~\ref{ex:length_faces}) and that of $\{\bm{A}\in\bbS^n \mid \langle\bm{A},\bm{J}_n\rangle \ge 0\}$ is $2$,
it follows from \eqref{eq:CP_Ln} and Lemma~\ref{lem:length_intersections_two_cones} that $\ell_{\CP(\bbL^n)} \le n+2$.
Thus, we have the desired result.
\end{proof}

\begin{proposition}\label{prop:dpcpl}
For $n\ge 2$, we have $\ell_{\rm poly}(\CP(\mathbb{L}^n)) = n-1$.
\end{proposition}

\begin{proof}
By Lemma~\ref{lem:bound_lpolyK_lK}, we have $\ell_{\rm poly}(\CP(\mathbb{L}^n)) \le \ell_{\CP(\bbL^n)} -2 = n$.
For the sake of obtaining a contradiction, we suppose that $\ell_{\rm poly}(\CP(\mathbb{L}^n)) = n$.
Then there exists a chain $\bbF_n \subsetneq \cdots \subsetneq \bbF_1$ of faces of $\CP(\bbL^n)$ such that $\bbF_n$ is not polyhedral.
Note that $\bbF_1 = \CP(\bbL^n)$.
Let $i^*$ be the maximum among $\{1,\dots,n\}$ such that there exists $\bm{A} \in \bbF_i$ with $\langle\bm{A},\bm{J}_n\rangle > 0$.
Then the dimensions of $\bbF_1,\dots,\bbF_{i^*}$ are in $\{T_n,\dots,T_1\}$ by the discussion in Section~\ref{subsec:face_CP_ge} and those of $\bbF_{i^*+1},\dots,\bbF_n$ are in $\{T_n-1,\dots,T_2-1,1,0\}$ by that in Section~\ref{subsec:face_CP_eq}.
Therefore, we have $\dim\bbF_{i^*} \le T_{n+1-i^*}$, $\dim\bbF_{i^*+1} \le T_{n+1-i^*}-1$, and $\dim\bbF_n \le T_{n+1-i^*-(n-i^*-1)} - 1 = 2$.
This implies that $\bbF_n$ is polyhedral, which is a contradiction.
Therefore, we see that $\ell_{\rm poly}(\CP(\mathbb{L}^n)) \le n-1$.

To prove $\ell_{\rm poly}(\CP(\mathbb{L}^n)) \ge n-1$, it is sufficient to show that the face $\CP(\bbX_2 \cap \bbL^n)$ in \eqref{eq:cp_longest_chain} is not polyhedral.
By $\dim\bbX_2 = 2$ and $\bbX_2 \cap \setint\bbL^n \neq \emptyset$, Lemma~\ref{lem:basis_in_partialK} implies that there exist $\bm{x},\bm{y} \in \bbX_2 \cap \partial\bbL^n$ such that $x_1 = y_1 = 1$ and $\{\bm{x},\bm{y}\}$ is a basis for $\bbX_2$.
Then we have that
\begin{equation}
(\bbX_2 \cap \bbL^n) \setminus \{\bm{0}\} = \{\alpha\bm{x}+\beta\bm{y} \mid \alpha,\beta \ge 0 \text{ and } (\alpha,\beta) \neq (0,0)\}. \label{eq:Ln_cap_W2}
\end{equation}
Now, we assume that $\CP(\bbX_2 \cap \bbL^n)$ is polyhedral.
Since polyhedral cones are finitely generated, there exist $\bm{Z}_1,\dots,\bm{Z}_m \in \bbS^n$ such that the extreme rays of $\CP(\bbX_2 \cap \bbL^n)$ are $\bbR_+\bm{Z}_1,\dots,\bbR_+\bm{Z}_m$.
Note that by Lemma~\ref{lem:exp_ext_closed_cone}, the generators of extreme rays of $\CP(\bbX_2 \cap \bbL^n)$ are the matrices $\bm{z}^{\otimes 2}$ with $\bm{z}\in (\bbX_2 \cap \bbL^n) \setminus \{\bm{0}\}$.
Therefore, for each $i=1,\dots,m$, there exists $\bm{z}_i \in (\bbX_2 \cap \bbL^n) \setminus \{\bm{0}\}$ such that $\bm{Z}_i =  \bm{z}_i^{\otimes 2}$.

By scaling $\bm{z}_i$ if necessary, \eqref{eq:Ln_cap_W2} implies that $\bm{z}_i$ can be written as a convex combination of $\bm{x}$ and $\bm{y}$, i.e., $\bm{z}_i = t_i\bm{x} + (1-t_i)\bm{y}$ for some $t_i\in [0,1]$.
Take $t \in [0,1] \setminus \{t_1,\dots,t_m\}$ and let $\hat{\bm{z}} \coloneqq t\bm{x} + (1-t)\bm{y}$.
Since $\hat{\bm{z}} \in (\bbX_2 \cap \bbL^n)\setminus \{\bm{0}\}$, the matrix $\hat{\bm{z}}^{\otimes 2}$ generates an extreme ray of $\CP(\bbX_2 \cap \bbL^n)$.
Therefore, there exist $i \in \{ 1,\dots,m\}$ and $c > 0$ such that $\hat{\bm{z}}^{\otimes 2} = c\bm{z}_i^{\otimes 2}$, i.e.,
\begin{equation*}
\begin{pmatrix}
1 & \hat{\bm{z}}_{2:n}^\top\\
\hat{\bm{z}}_{2:n} & \hat{\bm{z}}_{2:n}^{\otimes 2}
\end{pmatrix} = c\begin{pmatrix}
1 & (\bm{z}_i)_{2:n}^\top\\
(\bm{z}_i)_{2:n} & (\bm{z}_i)_{2:n}^{\otimes 2}
\end{pmatrix}.
\end{equation*}
This implies that $\hat{\bm{z}}_{2:n} = (\bm{z}_i)_{2:n}$.
Combining this with $x_1 = y_1 = 1$ yields $\hat{\bm{z}} = \bm{z}_i$, i.e., $t_i\bm{x} + (1-t_i)\bm{y} = t\bm{x} + (1-t)\bm{y}$.
Since $\bm{x}$ and $\bm{y}$ are linearly independent, we obtain $t = t_i$, which is a contradiction.
Hence, $\CP(\bbX_2 \cap \bbL^n)$ is not polyhedral.
\end{proof}

\section{Extensions to general copositive and completely positive cones}\label{sec:implication}
In this section, using the results shown in the previous sections,
we discuss the possibility of extending the results shown in the existing studies~\cite{Dickinson2011,NL2025} and this paper to general cones.
We examine some plausible conjectures and counter-examples.

Dickinson showed that the matrix $\bm{e}_i\bm{e}_j^\top +\bm{e}_j\bm{e}_i^\top$ generates an exposed ray of $\COP(\bbR_+^n)$ if $i \neq j$~\cite[Theorem~\mbox{4.6.\Rnum{1}}]{Dickinson2011} and a non-exposed extreme ray if $i = j$~\cite[Theorem~\mbox{4.4}]{Dickinson2011}.
For the case where the underlying cone is symmetric, we previously showed that an analogous result holds for the case $i = j$~\cite[Theorem~\mbox{3.3}]{NL2025}.
Specifically, we showed that for any $\bm{c}$ generating an extreme ray of a symmetric cone $\bbK$ of dimension greater than or equal to $2$ in the Euclidean space $\bbR^n$, the matrix $\bm{c}^{\otimes 2}$ generates a non-exposed extreme ray of $\COP(\bbK)$.
Based on these results, it is natural to conjecture whether for any $\bm{c}$ and $\bm{d}$ that generate extreme rays of a symmetric cone $\bbK$ (or a general closed convex cone) in $\bbR^n$ and are orthogonal to each other, the matrix $\bm{c}\bm{d}^\top + \bm{d}\bm{c}^\top$ generates an exposed ray of $\COP(\bbK)$.
However, this conjecture is not true.
Actually, as shown in the following example, the matrix $\bm{c}\bm{d}^\top + \bm{d}\bm{c}^\top$ does not necessarily generate an extreme ray of $\COP(\bbK)$.

\begin{example}\label{ex:not_face_cd_dc}
For any $\bm{v}\in\mathbb{R}^{n-1}$ with $\|\bm{v}\| = 1$, by Lemma~\ref{lem:soc_faces}, the two vectors $(1;\bm{v})$ and $(1;-\bm{v})$ generate extreme rays of the second-order cone $\bbL^n$ and are orthogonal to each other.
Using these vectors, we construct the following matrix:
\begin{equation}
(1;\bm{v})(1;-\bm{v})^\top + (1;-\bm{v})(1;\bm{v})^\top = 2\begin{pmatrix}
1 & \\
 & -\bm{v}^{\otimes 2}
\end{pmatrix}. \label{eq:counter_example_ray}
\end{equation}
By Proposition~\ref{prop:COP_ext}, a matrix that is not positive semidefinite and generates an extreme ray of $\COP(\mathbb{L}^n)$ must be a positive multiple of $\bm{J}_n$.
The rank of $\bm{J}_n$ is $n$, whereas that of the matrix in \eqref{eq:counter_example_ray} is $2$, which is a contradiction if $n \ge 3$.
Thus, the ray generated by the matrix in \eqref{eq:counter_example_ray} is not a face when $n \ge 3$.
\end{example}

Moving on, we showed in Proposition~\ref{prop:COP_ext} that $\bm{x}^{\otimes 2}$ generates an exposed ray of $\COP(\bbL^n)$ for every $\bm{x} \in \bbR^n \setminus (\bbL^n \cup (-\bbL^n))$.
Previously, Dickinson showed in \cite[Theorem~\mbox{4.6.\Rnum{2}}]{Dickinson2011} that $\bm{x}^{\otimes 2}$ generates an exposed ray of $\COP(\bbR_+^n)$ for every $\bm{x} \in \bbR^n \setminus (\bbR_+^n \cup (-\bbR_+^n))$.
Inspired by these facts, we obtain the following result.

\begin{proposition}\label{prop:x_neq_Kd_cup_minusKd}
Let $\bbK \subseteq \bbR^n$ be a closed convex cone with nonempty interior and let $\bm{x}\in \bbR^n \setminus (\bbK^* \cup (-\bbK^*))$.
Then the set $\bbR_+\bm{x}^{\otimes 2}$ is an exposed ray of $\COP(\bbK)$.
\end{proposition}

\begin{proof}
We first show that $\{\bm{x}\}^\perp \cap \setint\bbK \neq \emptyset$.
Since $\bm{x}\not\in \bbK^*$, it follows from the definition of $\bbK^*$ that there exists $\bm{y}_- \in \bbK$ satisfying $\bm{x}^\top\bm{y}_- < 0$.
Similarly, $\bm{x}\not\in -\bbK^*$ implies that there exists $\bm{y}_+ \in \bbK$ satisfying $\bm{x}^\top\bm{y}_+ > 0$.
Taking a strict convex combination of $\bm{y}_+$ and an element of $\setint\bbK$ if necessary, by the convexity of $\bbK$, we may assume that $\bm{y}_+ \in \setint\bbK$; see \cite[Theorem~\mbox{6.1}]{Rockafellar1970}.
Then
\begin{equation*}
\bm{y} \coloneqq \frac{\bm{x}^\top\bm{y}_+}{-\bm{x}^\top\bm{y}_- + \bm{x}^\top\bm{y}_+}\bm{y}_- + \frac{-\bm{x}^\top\bm{y}_-}{-\bm{x}^\top\bm{y}_- + \bm{x}^\top\bm{y}_+}\bm{y}_+
\end{equation*}
is a strict convex combination of $\bm{y}_- \in \bbK$ and $\bm{y}_+ \in \setint\bbK$, so we see that $\bm{y} \in \setint\bbK$.
In addition, we have $\bm{x}^\top\bm{y} = 0$.
Therefore, we obtain $\bm{y} \in \{\bm{x}\}^\perp \cap \setint\bbK$.

Since $\{\bm{x}\}^\perp \cap \setint\bbK \neq \emptyset$, Lemma~\ref{lem:basis_in_K} implies that there exists a basis $\calV$ for $\{\bm{x}\}^\perp$ such that $\calV \subseteq \{\bm{x}\}^\perp \cap \setint\bbK$.
Hence, it follows from \eqref{enum:AinCOP_perp} of Lemma~\ref{lem:HV} that $\COP(\bbK) \cap \{\bm{H}(\calV)\}^\perp = \bbR_+\bm{x}^{\otimes 2}$, which is an exposed ray of $\COP(\bbK)$.
\end{proof}

Next, we consider the possibility of extending parts of Theorem~\ref{thm:COP_sym_face_simplified} to general cones.
In particular, Theorem~\ref{thm:COP_sym_face_simplified} and \cite[Theorem~3.3]{NL2025} contain the fact that extreme rays of a symmetric cone $\bbK$ correspond to non-exposed extreme rays of  $\COP(\bbK)$ through the map that takes $\bm{x}$ to $\bbR_+\bm{x}^{\otimes 2}$.
With the aid of  Proposition~\ref{prop:x_neq_Kd_cup_minusKd}, we will see that this is not true in general.

\begin{example}\label{ex:counter_ex_general_K}
Let $\bbK \coloneqq \{\bm{x} \in \bbR^2 \mid x_2 \ge 0,\ x_1 + x_2 \ge 0\}$, which is a closed  convex cone with nonempty interior in $\bbR^2$.
We have
\begin{equation*}
\bbK^* \cup (-\bbK^*)=  \{y_1(0,1)^\top + y_2(1,1)^\top \mid \bm{y} \in \bbR_+^2 \cup (-\bbR_+^2)\}.
\end{equation*}
We observe that $\bm{e}_1 \in \Ext\bbK$. However, since the set $\bbK^* \cup (-\bbK^*)$ does not contain $\bm{e}_1$, it follows from Proposition~\ref{prop:x_neq_Kd_cup_minusKd} that the set $\bbR_+\bm{e}_1^{\otimes 2}$ is an exposed ray of $\COP(\bbK)$.
\end{example}

In Section~\ref{sec:CP}, we observed that $\CP(\bbX \cap \bbL^n)$ and $\CP(\bbX \cap \partial\bbL^n)$ are faces of $\CP(\bbL^n)$ for any subspace $\bbX$ of $\bbR^n$ (and vice versa).
It is a natural question whether these results also hold more generally.
In what follows, we show that the former holds but the latter does not.

\begin{proposition}\label{prop:exposed_face_CP_K_cap_X}
Let $\bbK$ be a closed cone and let $\bbX$ be a subspace of $\bbR^n$.
Then $\CP(\bbX \cap \bbK)$ is an exposed face of $\CP(\bbK)$.
If $\bbK$ is also convex and $\dim(\bbX \cap \bbK) = d$, we have $\dim\CP(\bbX \cap \bbK) = T_d$.
\end{proposition}

\begin{proof}
Let $\calV$ be a basis for $\bbX^\perp$.
Then it follows from \eqref{enum:AinCP_perp} of Lemma~\ref{lem:HV} that $\CP(\bbK) \cap \{\bm{H}(\calV)\}^\perp = \CP(\bbX \cap \bbK)$, which is an exposed face of $\CP(\bbK)$.

Next, we assume that $\bbK$ is  convex and $\dim(\bbX \cap \bbK)=d$ holds.
Since $\bbX \cap \bbK$ is a convex cone of dimension $d$, there exist $\bm{x}_1,\dots,\bm{x}_d \in \bbX \cap \bbK$ such that they are linearly independent.
Then, by Lemma~\ref{lem:linearly_independent_matrices}, $(\bm{x}_i +  \bm{x}_j)^{\otimes 2}$ ($1\le i\le j\le d$) are linearly independent.
Since each $(\bm{x}_i +  \bm{x}_j)^{\otimes 2}$ belongs to $\CP(\bbX \cap \bbK)$, it follows that $\dim\CP(\bbX \cap \bbK) \ge T_d$.
In addition, since $\CP(\bbX \cap \bbK) \subseteq \CP(\setspan(\bbX \cap \bbK)) = \bbS_+(\setspan(\bbX \cap \bbK))$ and since the latter is linearly isomorphic to $\bbS_+^d$, it follows that $\dim\CP(\bbX \cap \bbK) \le T_d$.
Therefore, we obtain $\dim\CP(\bbX \cap \bbK) = T_d$.
\end{proof}

The following example shows that $\CP(\bbX \cap \partial\bbK)$ is not necessarily a face of $\CP(\bbK)$ even if $\bbK$ is an irreducible symmetric cone.\footnote{A symmetric cone is called \emph{irreducible} if it cannot be represented as the direct sum of two nonzero symmetric cones.
The second-order cone $\bbL^n$ is irreducible if $n \ge 3$~\cite[Corollary~\mbox{\RNum{4}.1.5}]{FK1994}.}
To provide such an example, we utilize the vectorized positive semidefinite cones.
We define $\svec\colon\bbS^n \to \bbR^{T_n}$ as
\begin{equation*}
\svec(\bm{X}) \coloneqq (X_{1,1},\sqrt{2}X_{1,2},X_{2,2},\dots,\sqrt{2}X_{1,n},\dots,\sqrt{2}X_{n-1,n},X_{n,n})^\top,
\end{equation*}
that is, we are vectorizing $\bm{X}$ column by column in such a way that $\svec(\bm{X})^\top \svec(\bm{Y}) = \langle\bm{X}, \bm{Y} \rangle$ holds for $\bm{X}, \bm{Y} \in \bbS^n$.
Then $\svec(\cdot)$ is an isomorphism and, since $\bbS_+^n$ is irreducible~\cite[Chapter~\RNum{5}]{FK1994}, $\svec(\bbS_+^n)$ is also an irreducible symmetric cone.
For notational convenience, we write $\bm{X}^{\otimes 2}$ for the $T_n\times T_n$ symmetric matrix $\svec(\bm{X})^{\otimes 2}$ whose rows and columns are indexed by $ij$ with $1\le i\le j \le n$, respectively.\footnote{Note that the notation $\bm{X}^{\otimes 2}$ differs from the Kronecker product of $\bm{X}$ with itself.}
For instance, a matrix $\bm{X}\in \bbS^3$ induces the $6\times 6$ matrix
\begin{multline*}
\bm{X}^{\otimes 2} =\\
\begin{blockarray}{ccccccc}
11 & 12 & 22 & 13 & 23 & 33 & \\
\begin{block}{(cccccc)c}
X_{11}^2 & \sqrt{2}X_{11}X_{12} & X_{11}X_{22} & \sqrt{2}X_{11}X_{13} & \sqrt{2}X_{11}X_{23} & X_{11}X_{33} & 11\\
  & 2X_{12}^2 & \sqrt{2}X_{12}X_{22} & 2X_{12}X_{13} & 2X_{12}X_{23} & \sqrt{2}X_{12}X_{33} & 12\\
  &   & X_{22}^2 & \sqrt{2}X_{22}X_{13} & \sqrt{2}X_{22}X_{23} & X_{22}X_{33} & 22\\
  &   &   & 2X_{13}^2 & 2X_{13}X_{23} & \sqrt{2}X_{13}X_{33} & 13\\
  &   &   &   & 2X_{23}^2 & \sqrt{2}X_{23}X_{33} & 23\\
  &   &   &   &   & X_{33}^2 & 33\\
\end{block}
\end{blockarray},
\end{multline*}
where the strictly lower triangular portion can be omitted because of the symmetry of the matrix $\bm{X}^{\otimes 2}$.
For a matrix $\bm{X}\in \bbS^n$, the matrix $\bm{X}^{\otimes 2} \in \bbS^{T_n}$ is of rank $1$ in the sense that it is generated by the vector $\svec(\bm{X})$.
Note that this does not mean that the matrix $\bm{X}$ is of rank $1$.

\begin{example}\label{ex:not_face_CP_partialK_cap_X}
Consider the following matrix composed of three rank-$1$ matrices in $\CP(\partial\svec(\bbS_+^3))$:
\begin{align}
&\begin{pmatrix}
1 & 0 & 0\\
  & 1 & 0\\
  &   & 0
\end{pmatrix}^{\otimes 2} + \begin{pmatrix}
0 & 0 & 0\\
  & 1 & 0\\
  &   & 1
\end{pmatrix}^{\otimes 2} + \begin{pmatrix}
1 & 0 & 0\\
  & 0 & 0\\
  &   & 1
\end{pmatrix}^{\otimes 2} \nonumber\\
&\quad = \begin{blockarray}{ccccccc}
11 & 12 & 22 & 13 & 23 & 33 & \\
\begin{block}{(cccccc)c}
2 & 0 & 1 & 0 & 0 & 1 & 11\\
  & 0 & 0 & 0 & 0 & 0 & 12\\
  &   & 2 & 0 & 0 & 1 & 22\\
  &   &   & 0 & 0 & 0 & 13\\
  &   &   &   & 0 & 0 & 23\\
  &   &   &   &   & 2 & 33\\
\end{block}
\end{blockarray} \label{eq:PSD_counter_example}\\
&\quad \in \CP(\partial\svec(\bbS_+^3)). \nonumber
\end{align}
The matrix in the right-hand side of \eqref{eq:PSD_counter_example} has another representation as the sum of rank-$1$ matrices in $\CP(\svec(\bbS_+^3))$, i.e.,
\begin{equation}
\eqref{eq:PSD_counter_example} = \begin{pmatrix}
1 & 0 & 0\\
  & 1 & 0\\
  &   & 1
\end{pmatrix}^{\otimes 2} + \begin{pmatrix}
1 & 0 & 0\\
  & 0 & 0\\
  &   & 0
\end{pmatrix}^{\otimes 2} + \begin{pmatrix}
0 & 0 & 0\\
  & 1 & 0\\
  &   & 0
\end{pmatrix}^{\otimes 2} + \begin{pmatrix}
0 & 0 & 0\\
  & 0 & 0\\
  &   & 1
\end{pmatrix}^{\otimes 2}. \label{eq:PSD_counter_example_another}
\end{equation}
The matrix
\begin{equation*}
\begin{pmatrix}
1 & 0 & 0\\
  & 1 & 0\\
  &   & 1
\end{pmatrix}^{\otimes 2} =
\begin{blockarray}{ccccccc}
11 & 12 & 22 & 13 & 23 & 33 & \\
\begin{block}{(cccccc)c}
1 & 0 & 1 & 0 & 0 & 1 & 11\\
  & 0 & 0 & 0 & 0 & 0 & 12\\
  &   & 1 & 0 & 0 & 1 & 22\\
  &   &   & 0 & 0 & 0 & 13\\
  &   &   &   & 0 & 0 & 23\\
  &   &   &   &   & 1 & 33\\
\end{block}
\end{blockarray} \in \CP(\svec(\bbS_+^3)),
\end{equation*}
denoted by $\bm{A}$, appearing in the first term in \eqref{eq:PSD_counter_example_another} does not belong to $\CP(\partial\svec(\bbS_+^3))$.
To show this by contradiction, we assume that $\bm{A} \in \CP(\partial\svec(\bbS_+^3))$.
Then, since the rank of $\bm{A}$ is $1$, there exists $\bm{X} \in \partial\bbS_+^3$ such that $\bm{A} = \bm{X}^{\otimes 2}$.
First, for any $i = 1,2,3$, by $1 = A_{ii,ii} = X_{ii}^2$ and the positive semidefiniteness of $\bm{X}$, we have $X_{ii} = 1$.
Second, for any $(i,j)$ with $1 \le i < j \le 3$, since $0 = A_{ij,ij} = 2X_{ij}^2$, it follows that $X_{ij} = 0$.
Combining them, we see that $\bm{X}$ is the identity matrix, which contradicts the assumption that $\bm{X} \in \partial\bbS_+^3$.
Therefore, $\bm{A}\not\in \CP(\partial\svec(\bbS_+^3))$ holds and $\CP(\partial\svec(\bbS_+^3))$ is not a face of $\CP(\svec(\bbS_+^3))$.
\end{example}

Finally, we revisit the problem of computing the length of a longest chain of faces and the distance to polyhedrality of $\COP(\bbK)$ and $\CP(\bbK)$.
These quantities are now known for $\bbK = \bbR_+^n$ \cite{Nishijima2024} and for $\bbK = \bbL^n$ (Propositions~\ref{prop:length_COPLn}, \ref{prop:dpcopl}, \ref{prop:lcpl} and \ref{prop:dpcpl}).
The case where $\bbK$ is a general symmetric cone has proved elusive so far and the case of general $\bbK$ seems completely out of reach at the moment.
Nevertheless, if a symmetric cone $\bbK$ satisfies $\dim \bbK \geq 2$, then, for some $n \ge 2$, $\COP(\bbK)$ and $\CP(\bbK)$ have faces isomorphic to $\COP(\bbL^n)$ and $\CP(\bbL^n)$, respectively.
This can be leveraged to aid in computing (but not completely determining) these two parameters, see the discussion after Corollary~\ref{cor:COP_sym_face} in Appendix~\ref{apdx:EJA}.

\vspace{0.5cm}
\noindent
{\bf Acknowledgments}

The first author is supported by JSPS Grant-in-Aid for JSPS Fellows JP22KJ1327 and JSPS Grant-in-Aid for Research Activity Start-up JP25K23344.
The second author is supported by JSPS Grant-in-Aid for Early-Career Scientists JP23K16844 and JST ASPIRE Grant Number JPMJAP2520.

\begin{appendices}
\begin{section}{
Faces of copositive and completely positive cones over symmetric cones}\label{apdx:EJA}
This appendix presents extensions of some of the results shown in \cite{NL2025}, thereby clarifying the motivation to study the facial structure of copositive and completely positive cones over a second-order cone.
The results are presented through the lens of \emph{Euclidean Jordan algebras}, which offer a one-to-one correspondence to symmetric cones.
The terminology used here is standard and follows  \cite{NL2025,FK1994}.

In order to properly discuss the results, we extend the definitions of copositive and completely positive cones exhibited in \eqref{eq:COP} and \eqref{eq:CP} from matrices to linear mappings.
Let $(\bbE,\bullet)$ be a finite-dimensional real inner product space and $\calS(\bbE)$ be the space of self-adjoint linear transformations on $\bbE$.
For a closed cone $\bbK$ in the space $\bbE$, the copositive cone and the completely positive cone over $\bbK$ are defined as
\begin{equation}
\COP(\bbK) \coloneqq \{\calA \in \calS(\bbE) \mid  x\bullet \mathcal{A}(x) \ge 0 \text{ for all $x\in \mathbb{K}$}\} \label{eq:COP_transformation}
\end{equation}
and
\begin{equation}
\CP(\bbK)\coloneqq\left\{\sum_{i=1}^k a_i\otimes a_i \relmiddle|\text{$k$ is a positive integer and } a_i \in \bbK \text{ for all $i = 1,\dots,k$}\right\},\label{eq:CP_transformation}
\end{equation}
respectively.\footnote{Note that the tensor space $\bbE\otimes \bbE$ can be identified with the space of linear mappings on $\bbE$ through the natural isomorphism that maps every element $a\otimes b$ to the mapping $x \mapsto (b\bullet x)a$.}
Here, we use the notation $a_i \otimes a_i$ instead of $a_i^{\otimes 2}$ to emphasize that $\bbE$ is not necessarily of the form $\bbR^n$ for some $n$.
That said, if  $\bbE = \bbR^n$ holds, then the copositive cone in \eqref{eq:COP_transformation} and the completely positive cone in \eqref{eq:CP_transformation} can be identified with those defined in \eqref{eq:COP} and \eqref{eq:CP}, respectively.
As with the duality between \eqref{eq:COP} and \eqref{eq:CP}, the copositive cone in \eqref{eq:COP_transformation} and the completely positive cone in \eqref{eq:CP_transformation} are dual to each other with respect to the inner product on $\calS(\bbE)$, which is denoted by $\langle\cdot,\cdot\rangle$, induced by the inner product $\bullet$ on $\bbE$.
See \cite[page~237]{Satake1975} and \cite[page~838]{NN2024_Approximation} for the construction of the inner product on $\calS(\bbE)$.

As mentioned above, symmetric cones and Euclidean Jordan algebras have a one-to-one correspondence.
For any symmetric cone $\bbK$ in a finite-dimensional real inner product space $(\bbE,\bullet)$, there exists a bilinear product $\circ\colon\bbE\times \bbE\to \bbE$ such that $(\bbE,\circ,\bullet)$ is a Euclidean Jordan algebra and $\bbK$ is the cone of squares in $\bbE$, i.e., the cone $\bbE_+ \coloneqq \{x\circ x \mid x\in \bbE\}$~\cite[Theorem~\mbox{\RNum{3}.3.1}]{FK1994}.
Conversely, for any Euclidean Jordan algebra $(\bbE,\circ,\bullet)$, the cone $\bbE_+$ is a symmetric cone in $\bbE$~\cite[Theorem~\mbox{\RNum{3}.2.1}]{FK1994}.

In what follows, we use $(\bbE,\circ,\bullet)$, or simply $\bbE$ to denote a Euclidean Jordan algebra of rank $r$, where $\bbE$ is a finite-dimensional real vector space, $\circ\colon \bbE \times \bbE \to \bbE$ is a bilinear product, and $\bullet\colon \bbE\times \bbE \to \bbR$ is an associative inner product.
We write $e$ for the identity element with respect to the product $\circ$.
For an idempotent $c \in \bbE$, we define $\bbE(c,1) \coloneqq \{x \in \bbE \mid c\circ x =  x\}$.
The space $\bbE(c,1)$ is a Euclidean Jordan subalgebra of $\bbE$~\cite[Proposition~\mbox{\RNum{4}.1.1}]{FK1994}, so the cone $\bbE(c,1)_+$ is a symmetric cone in the space $\bbE(c,1)$.

\begin{theorem}\label{thm:COP_sym_face}
Let $c$ be an idempotent in a Euclidean Jordan algebra $\bbE$.
In addition, let $\calP_{\bbE(c,1)}\colon\bbE \to \bbE(c,1)$ be the linear mapping that projects every element in $\bbE = \bbE(c,1)^\perp \oplus \bbE(c,1)$ onto the subspace $\bbE(c,1)$ and let $\calP_{\bbE(c,1)}^*$ be its adjoint.
Then the set
\begin{equation}
\{0\} \oplus \COP(\bbE(c,1)_+) \coloneqq \{\calP_{\bbE(c,1)}^* \calA \calP_{\bbE(c,1)} \mid \calA \in \COP(\bbE(c,1)_+)\} \label{eq:COP_sym_face}
\end{equation}
is a $T_{\dim\bbE(c,1)}$-dimensional face of $\COP(\bbE_+)$.
Furthermore, the face in \eqref{eq:COP_sym_face} is exposed if and only if $c$ is either the identity element $e$ or the zero.
\end{theorem}

The notation $\{0\} \oplus \COP(\bbE(c,1)_+)$ used in \eqref{eq:COP_sym_face} comes from the fact that it can be regarded as the set consisting of matrices whose lower-right block is an element of $\COP(\bbE(c,1)_+)$ and whose other blocks are zeros.
See also \cite[Section~2.2]{NL2025} for the rationale for this notation.

\begin{proof}
It was shown in \cite[Lemma~\mbox{3.1}]{NL2025} that the set in \eqref{eq:COP_sym_face} is a face of $\COP(\bbE_+)$.
In addition, the cone in \eqref{eq:COP_sym_face} is linearly isomorphic to $\COP(\bbE(c,1)_+)$ through the isomorphism $\psi\colon \calS(\bbE(c,1)) \to  \{\calP_{\bbE(c,1)}^* \calA \calP_{\bbE(c,1)} \mid \calA \in \calS(\bbE(c,1))\}$ such that
\begin{equation*}
\psi(\calA) = \calP_{\bbE(c,1)}^* \calA \calP_{\bbE(c,1)}.
\end{equation*}
Since $\COP(\bbE(c,1)_+)$ is full-dimensional in the space $\calS(\bbE(c,1))$, the dimension of $\COP(\bbE(c,1)_+)$ is $T_{\dim \bbE(c,1)}$.

If $c = e$, then the face in \eqref{eq:COP_sym_face} is equal to $\COP(\bbE_+)$, which is indeed exposed.
If $c = 0$, then the face in \eqref{eq:COP_sym_face} is equal to $\{0\}$, which is also an exposed face of $\COP(\bbE_+)$.
In what follows, we assume that $c$ is neither $e$ nor $0$ and assume, for the sake of obtaining a contradiction, that the face in \eqref{eq:COP_sym_face} is exposed.
Then there exists $\calH \in \CP(\bbE_+)$ such that
\begin{equation}
\{0\} \oplus \COP(\bbE(c,1)_+) =  \COP(\bbE_+) \cap \{\calH\}^\perp. \label{eq:COP_sym_face_exposed}
\end{equation}
We decompose the idempotent $c$ into primitive idempotents $c_1,\dots,c_p \in \bbE(c,1)$ that are orthogonal to each other.
Note that $p \ge 1$ since $c \neq 0$.
We define $\calC \coloneqq \sum_{i=1}^pc_i\otimes c_i$
and $\calC$ can be regarded as either an element in $\calS(\bbE(c,1))$ or in $\calS(\bbE)$.
Since $\calC$ is positive semidefinite, we have
\begin{equation*}
\calC = \calP_{\bbE(c,1)}^* \calC \calP_{\bbE(c,1)} \in \{0\} \oplus \COP(\bbE(c,1)_+).
\end{equation*}
Therefore, by \eqref{eq:COP_sym_face_exposed}, we see that $\calC \in \COP(\bbE_+) \cap \{\calH\}^\perp$.
Now, since $\calH \in \CP(\bbE_+)$, there exist $h_1,\dots,h_k\in \bbE_+$ such that $\calH = \sum_{j=1}^kh_j\otimes h_j$.
Then it follows from $\calC \in \{\calH\}^\perp$ that
\begin{equation*}
0 = \langle\calC,\calH\rangle = \sum_{i=1}^p\sum_{j=1}^k(c_i\bullet h_j)^2.
\end{equation*}
This implies that
\begin{equation}
c_i\bullet h_j = 0 \label{eq:ci_dot_hj_0}
\end{equation}
for all $i = 1,\dots,p$ and $j=1,\dots,k$.

The convex cone $\bbE(c,1)^\perp \cap \bbE_+$ has dimension greater than zero since $c \neq e$.
Let $d$ be a nonzero element in $\bbE(c,1)^\perp \cap \bbE_+$.
Using $d$, we define $\calB \coloneqq \sum_{i=1}^p(c_i\otimes d + d\otimes c_i)$.
For any $x\in \bbE_+$, it follows from the self-duality of $\bbE_+$ that
\begin{equation*}
x\bullet \calB(x) = 2\sum_{i=1}^p(c_i\bullet x)(d\bullet x) \ge 0,
\end{equation*}
so we have $\calB \in \COP(\bbE_+)$.
In addition, it follows from \eqref{eq:ci_dot_hj_0} that
\begin{equation*}
\langle\calB,\calH\rangle = 2\sum_{i=1}^p\sum_{j=1}^k(c_i\bullet h_j)(d\bullet h_j) = 0.
\end{equation*}
Therefore, we see from \eqref{eq:COP_sym_face_exposed} that $\calB \in \{0\} \oplus \COP(\bbE(c,1)_+)$.
This implies that there exists $\calA \in \COP(\bbE(c,1)_+)$ such that $\calB = \calP_{\bbE(c,1)}^* \calA \calP_{\bbE(c,1)}$.
Since $d \in \bbE(c,1)^\perp$, we obtain $\calB(d) = 0$.
However, by the definition of $\calB$, we have $\calB(d) = \sum_{i=1}^p\lVert d\rVert^2 c_i$, which is nonzero, a contradiction.
\end{proof}

Next, we phrase Theorem~\ref{thm:COP_sym_face} in geometric terms.
The key observation for that is the one-to-one correspondence between idempotents of a Euclidean Jordan algebra and faces of the underlying symmetric cone~\cite[Theorem~\mbox{3.1}]{GS2006}.
For any idempotent $c\in \bbE$, the cone $\bbE(c,1)_+$ is a face of the symmetric cone $\bbE_+$.
Converserly, for any face $\bbF$ of the symmetric cone $\bbE_+$, there exists an idempotent $c\in \bbE$ such that $\bbF = \bbE(c,1)_+$.
Then we obtain the following corollary, which implies  Theorem~\ref{thm:COP_sym_face_simplified}.

\begin{corollary}\label{cor:COP_sym_face}
Let $\bbK$ be a symmetric cone and $\bbF$ be a face of $\bbK$.
Then the set $\{0\}\oplus \COP(\bbF)$ is a $T_{\dim\bbF}$-dimensional face of $\COP(\bbK)$.
Furthermore, the face $\{0\}\oplus \COP(\bbF)$ is exposed if and only if the face $\bbF$ is either the symmetric cone $\bbK$ or the singleton $\{0\}$.
\end{corollary}

We note that Corollary~\ref{cor:COP_sym_face} strengthens  \cite[Theorem~\mbox{3.3}]{NL2025}.
For a $1$-dimensional face $\bbR_+c$ of a symmetric cone, the associated face $\{0\}\oplus \COP(\bbR_+c)$ obtained in Corollary~\ref{cor:COP_sym_face} can be written as $\bbR_+c\otimes c$, whose non-exposedness we have shown in \cite[Theorem~\mbox{3.3}]{NL2025}.
By contrast, in Corollary~\ref{cor:COP_sym_face}, we show the non-exposedness of the face $\{0\}\oplus \COP(\bbF)$ for \emph{any} face $\bbF$ being neither the symmetric cone itself nor $\{0\}$.
Also, if as in Section~\ref{subsec:COP_CP}, $\bbK$ is contained in some $\bbR^n$ and we express the elements of $\{0\}\oplus \COP(\bbF)$ using the standard basis, we arrive at the expression in \eqref{eq:COPF_cap_span}.

As pointed out in \cite[page~\mbox{13}]{NL2025}, since a face $\bbF$ of a symmetric cone is also a symmetric cone on its span, we can apply Corollary~\ref{cor:COP_sym_face} to the face $\{0\}\oplus \COP(\bbF)$, which is linearly isomorphic to $\COP(\bbF)$.
Consequently, from a chain
\begin{equation}
\bbF_l \subsetneq \cdots \subsetneq \bbF_1 \label{eq:chain_face_symcone}
\end{equation}
of faces of a symmetric cone $\bbK$, we obtain the following chain of faces of $\COP(\bbK)$:
\begin{equation*}
\{0\}\oplus \COP(\bbF_l) \subsetneq \cdots \subsetneq \{0\}\oplus \COP(\bbF_1),
\end{equation*}
where the strict inclusions follow from $\dim(\{0\}\oplus \COP(\bbF_i)) = T_{\dim\bbF_i}$ for each $i = 1,\dots,l$ and $\dim\bbF_l < \cdots < \dim\bbF_1$.
For a symmetric cone $\bbK$ associated with a Euclidean Jordan algebra of rank $r$, it is known that the cone $\bbK$ has a chain of faces with length $\ell_{\bbK} = r + 1$~\cite[Theorem~\mbox{14}]{IL2017}.
Therefore, for such a symmetric cone $\bbK$, we have
\begin{equation}
\ell_{\COP(\bbK)} \ge r + 1. \label{eq:ell_COP_symcone_bound}
\end{equation}
In addition, for a chain in \eqref{eq:chain_face_symcone} of faces of the symmetric cone $\bbK$ with length $l = r+1$, it follows from Corollary~\ref{cor:COP_sym_face} that the subface $\{0\}\oplus \COP(\bbF_r)$ is not an exposed face of $\{0\}\oplus \COP(\bbF_{r-1})$.
Hence, the face $\{0\}\oplus \COP(\bbF_{r-1})$ is not polyhedral, so we have
\begin{equation}
\ell_{\rm poly}(\COP(\bbK)) \ge r-1. \label{eq:ellpoly_COP_symcone_bound}
\end{equation}

We can see from Theorem~\ref{thm:COP_sym_face} that for any symmetric cone $\bbE_+$ of dimension greater than or equal to $2$, the copositive cone $\COP(\bbE_+)$ has a face linearly isomorphic to the copositive cone over a second-order cone.
Indeed, by $\dim\bbE_+ \ge 2$, the rank $r$ of $\bbE$ is greater than or equal to $2$.
Then there exists an idempotent $c\in \bbE$ of rank $2$, so that the rank of the Euclidean Jordan subalgebra $\bbE(c,1)$ is $2$; see \cite[Proposition~2]{Faybusovich2006} and \cite[Theorem~5]{LT2020}.
If the Euclidean Jordan subalgebra $\bbE(c,1)$ is not simple, it is linearly isomorphic to the direct sum of the $1$-dimensional Euclidean Jordan algebra $\bbR$ with itself whose product and inner product are multiplication defined for real numbers.
In such a case, the cone $\bbE(c,1)_+$ is linearly isomorphic to the $2$-dimensional second-order cone.
If the Euclidean Jordan subalgebra $\bbE(c,1)$ is simple, by the classification result of simple Euclidean Jordan algebras of rank $2$~\cite[Corollary~\mbox{\RNum{4}.1.5}]{FK1994}, it is linearly isomorphic to the Euclidean Jordan algebra $\bbR^n$ for some $n\ge 3$ whose product $\hat{\circ}$ and inner product $\hat{\bullet}$ are defined as
\begin{align*}
\bm{x}\mathbin{\hat{\circ}}\bm{y} &\coloneqq \begin{pmatrix}
\bm{x}^\top\bm{y}\\
x_1\bm{y}_{2:n} + y_1\bm{x}_{2:n}
\end{pmatrix},\\
\bm{x}\mathbin{\hat{\bullet}}\bm{y} &\coloneqq \bm{x}^\top\bm{y}
\end{align*}
for each $\bm{x},\bm{y}\in \bbR^n$, respectively.
In such a case, the cone $\bbE(c,1)_+$ is linearly isomorphic to the $n$-dimensional second-order cone.
To sum up, regardless of whether $\bbE(c,1)$ is simple or not, the cone $\bbE(c,1)_+$ is linearly isomorphic to a second-order cone $\bbL^n$ for some $n\ge 2$.
Then Theorem~\ref{thm:COP_sym_face} implies that the copositive cone $\COP(\bbE_+)$ has a face linearly isomorphic to $\COP(\bbL^n)$.
Therefore, studying the facial structure of $\COP(\bbL^n)$ provides further insight into that of copositive cones over general symmetric cones.
As an illustration of such implications, by computing the length of a longest chain of faces and the distance to polyhedrality of $\COP(\bbL^n)$, we can strengthen the lower bounds shown in \eqref{eq:ell_COP_symcone_bound} and \eqref{eq:ellpoly_COP_symcone_bound}.

Analogous results hold for completely positive cones.
For any symmetric cone $\bbK$ of dimension greater than or equal to $2$, it follows from Proposition~\ref{prop:clconv_cone_to_CP} that the completely positive cone $\CP(\bbK)$ has a face linearly isomorphic to $\CP(\bbL^n)$ for some $n \ge 2$.
This fact serves as one of the motivations for the present study.
\end{section}

\end{appendices}

\bibliographystyle{plainurl} 
\bibliography{NL24_1_ref} %

\end{document}